%% file: ms.tex
\documentclass[a4paper,10pt]{article}

\usepackage{amsmath}
\usepackage{amsthm}
\usepackage{amssymb}
\usepackage{amstext}
\usepackage{dsfont}
\usepackage{graphicx}
\usepackage{subfig}
\usepackage[dvips]{epsfig}
\usepackage{texdraw}
\usepackage{color}
\usepackage{url}
\usepackage{framed} 
\usepackage{tikz}
\usetikzlibrary{fit,positioning,arrows,shapes}
\usepackage{pgfplots}
\pgfplotsset{compat=1.11}
\usepackage{a4wide}
\usepackage{paralist}

\usepackage{tabularx}
\newcolumntype{L}[1]{>{\raggedright\arraybackslash}p{#1}} 
\newcolumntype{C}[1]{>{\centering\arraybackslash}p{#1}} 
\newcolumntype{R}[1]{>{\raggedleft\arraybackslash}p{#1}} 
\usepackage{array}
\usepackage{booktabs}
\usepackage{nicefrac}

\theoremstyle{plain}

\theoremstyle{definition}

\theoremstyle{remark}

\usepackage[ruled,vlined]{algorithm2e}

\makeatletter
\def\BState{\State\hskip-\ALG@thistlm}
\makeatother

\graphicspath{{bilder/}}

\DeclareSymbolFont{msbm}{U}{msb}{m}{n}
\DeclareMathSymbol{\N}{\mathalpha}{msbm}{'116}
\DeclareMathSymbol{\R}{\mathalpha}{msbm}{'122}

\newcommand{\be}{\begin{eqnarray}}
\newcommand{\ee}{\end{eqnarray}}
\newcommand{\Ind}{\mathds{1}}
\newcommand{\Div}{\operatorname{div}}

\newtheorem{theorem}{Theorem}[section]
\newtheorem{lemma}[theorem]{Lemma}

\newtheorem{definition}[theorem]{Definition}
\newtheorem{remark}[theorem]{Remark}

\numberwithin{equation}{section}

\newcommand{\sgn}{\operatorname{sgn}}

\begin{document}

\title{A pedestrian flow model with stochastic velocities: microscopic and macroscopic approaches} 

\author{S.\ G\"ottlich\footnotemark[1], \; S.\ Knapp\footnotemark[1], \; P.\ Schillen\footnotemark[1]}

\footnotetext[1]{University of Mannheim, Department of Mathematics, 68131 Mannheim, Germany (goettlich@uni-mannheim.de, stknapp@mail.uni-mannheim.de).}

\date{\today}

\maketitle

\begin{abstract}
\noindent
We investigate a stochastic model hierarchy for pedestrian flow.
Starting from a microscopic social force model, where the
pedestrians switch randomly between the two states stop-or-go,
we derive an associated macroscopic model of conservation law type.
Therefore we use a kinetic mean-field equation and introduce a new problem-oriented  
closure function. Numerical experiments are presented to compare
the above models and to show their similarities.
\end{abstract}



\input{introduction.tex}

\input{sec2.tex}

\input{numerics.tex}

\input{comp_results.tex}

\input{conclusion.tex}


\section*{Acknowledgments}
This work was financially supported by the DAAD project ``DAAD-PPP VR China'' (project ID: 57215936)
and the DFG grant GO 1920/4-1.

\bibliographystyle{siam}
\bibliography{ms}


\end{document}

%% file: introduction.tex
\section{Introduction}
\label{sec:1}
The modeling of crowd dynamics is a current research topic and provides a useful tool
for evacuation planning.
A good overview of the existing literature can be found in \cite{Bellomo2011, Degond2013}, 
where mainly two classes of modeling approaches (microscopic versus macroscopic) are distinguished. 
Microscopic pedestrian models typically rely on Newton-type dynamics as proposed e.g. in~\cite{HelbingMolnar1998, PiccoliTosin2009} 
while macroscopic models can be either derived via limiting processes \cite{Chen2016, Degond2013, Klar_2014} or 
phenomenologically, see e.g. \cite{Helbing1992, Hughes_2002, PiccoliTosin2009, PiccoliTosin2011}.
Starting from a microscopic level model extensions include for example vision cones \cite{Degond2013Vision}, shortest-path information \cite{Etikyala2014} and diffusion \cite{Degond2013, Etikyala2014}. 

In this work, we focus on a well-known phenomenon in crowds which is the sudden presence of non-moving people stopping immediately
due to external attractions or, more recently, new cell phone messages. 
This leads to significant changes in the individual walking velocities, rerouting actions and hence to local
bottlenecks depending on the crowd density. 
From reality we observe that people stop directly in front of attractions at a high probability
and therefore influence other individuals also to stop or keep on walking.
This random effect will be included as a stochastic process which
is dependent on the position of the pedestrians.
The main difference to existing pedestrian models with random effects is the motivation mentioned above and the mathematical 
representation of the stochastic dynamics.  
Similar problems have been considered in e.g.~\cite{Degond2013, Etikyala2014, Klar_2014}, where
a formulation of the dynamical system in terms of Brownian motion driven stochastic differential equations (SDEs) 
has been used to capture intrinsic random decisions in the velocity of pedestrians. 
However, this approach is not suitable to cover the stop-and-go behavior of pedestrians we have in mind. 
We will use a continuous-time Markov chain approach \cite{DegondRinghofer2007} to include the random stop-and-go behavior
by considering location-dependent switching rates. The overall goal is then to derive  
a model hierarchy for the microscopic and macroscopic pedestrian flow. 
In a first step, we develop and define the stochastic microscopic pedestrian model based on the social force model 
by Helbing and Moln\'{a}r~\cite{HelbingMolnar1998} and prove its existence, see section \ref{sec:2}. 
Then, the kinetic formulation of the microscopic model 
under the molecular chaos assumption in terms of measures is considered. The derivation   
of a reasonable closure function to state the associated macroscopic pedestrian model is non-standard and requires
some computational effort.
Section \ref{sec:3} deals with the numerical treatment of the microscopic and macroscopic equations. We present a stochastic simulation algorithm for the microscopic model and explain how numerical methods for hyperbolic conservation laws can be used to approximate the solution of the macroscopic model. 
For simulation purposes we present two examples in section \ref{sec:4} and 
compare qualitatively both approaches by defining reasonable performance measures. 
An additional experiment with a deterministic situation, i.e., no random velocity drops to zero,
will be also discussed. 

The following picture \ref{fig:Overview1} summarizes the main properties of the deterministic and stochastic modeling
approach for different levels of descriptions. It intends to give an idea on the main ingredients used throughout the sections
\ref{sec:2} to \ref{sec:4}.

\begin{figure}[htb!]
\centering
\begin{tikzpicture}[font=\footnotesize,every node/.style={node distance=2cm},%
    force1/.style={draw, fill=black!0,inner sep=2mm,text width=0.1\textwidth, align=flush center,%
    minimum height=1.2cm, font=\bfseries\footnotesize},
    force2/.style={draw, fill=black!0,inner sep=2mm,text width=0.37\textwidth, align=flush center,%
    minimum height=1.2cm, font=\bfseries\footnotesize},
    force3/.style={draw, fill=black!0,inner sep=2mm,text width=0.37\textwidth, align=flush center,%
    minimum height=1.2cm, font=\bfseries\footnotesize},
    force4/.style={draw, fill=black!0,inner sep=2mm,text width=0.37\textwidth, align=flush center,%
    minimum height=2cm, font=\footnotesize},
    force5/.style={draw, fill=black!0,inner sep=2mm,text width=0.37\textwidth, align=flush center,%
    minimum height=3cm, font=\footnotesize}]


  \node [force5](detmac){\vspace{0mm}
\begin{compactitem}
\item Scalar conservation law with advection part only
\item One velocity model: Pedestrians always move with the closure velocity \\
\end{compactitem}  
  };
  \node [force5,right=.5cm of detmac](stochmac){\vspace{0mm}
\begin{compactitem}
\item System of conservation laws with advection and reaction part 
\item Two velocities model: Pedestrians may change their velocities (stop or go) 
\end{compactitem}  
  };

  \node [force4, above = .0cm of detmac](detmic){\vspace{0mm}
\begin{compactitem}
\item Variables: position and velocities
\item Time-continuous Newton-type dynamics
\end{compactitem}  
  };
  \node [force4,right=.5cm of detmic](stochmic){\vspace{0mm}
\begin{compactitem}
\item Variables: position, velocities and status (stop or go)
\item Time-discrete stochastic process
\end{compactitem}  
  };

  \node [force2, above of=detmic](det){Deterministic};
  \node [force3, above of=stochmic](stoch){Stochastic};
  
  \node [force1, left=.5cm of detmic](mic){Micro{\-}scopic};
  \node [force1, left=.5cm of detmac](mac){Macro{\-}scopic};
  
  
  \path[-]
;
\end{tikzpicture}
\caption{Overview of the deterministic and stochastic model hierarchy}
\label{fig:Overview1}
\end{figure}
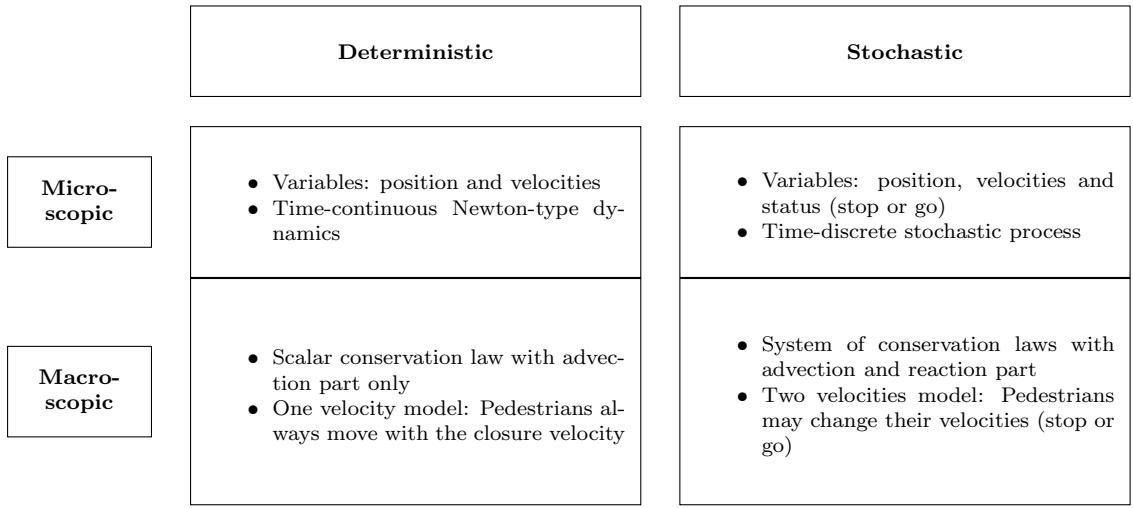

%% file: sec2.tex
\section{The Models}
\label{sec:2}
In this section, we derive a stochastic microscopic model and deduce a corresponding macroscopic model. First, we define the stochastic microscopic model based on a deterministic social force model~\cite{HelbingMolnar1998} and show its well-posedness. 
Second, we determine the time-evolution of the microscopic model and use a mean field assumption to get a kinetic model. 
This model is then simplified by an appropriate closure assumption to develop a macroscopic model, cf. 
existing literature on particle systems, e.g.\ \cite{DiluteGases,JABIN_2002, JuengelSemiconductor,Klar_2014}
and pedestrian flow models in e.g.\ \cite{Degond2013}.

\input{microscopic_model.tex}

\input{kinetic_model.tex}

\input{macroscopic_model.tex}

%% file: microscopic_model.tex
\subsection{Microscopic Model}
We use the ideas introduced by Helbing and Moln\'{a}r \cite{HelbingMolnar1998}
to describe the dynamical behavior of pedestrians. Let $N \in \N$ be the number of pedestrians and let 
\[x_i(t) = (x_i^{(1)}(t),x_i^{(2)}(t)) \in \R^2\] 
the position of pedestrian $i$ at time $t \in \R_{\geq 0}$. The velocity of pedestrian $i$ at time $t$ is given by 
\[\dot{x}_i(t) = v_i(t)=(v_i^{(1)}(t),v_i^{(2)}(t))  \in \R^2.\] In the following, we denote by
\[\vec{x}(t) = (x_1(t),\dots,x_N(t)) \in \R^{2N}\] and \[\vec{v}(t) = (v_1(t),\dots,v_N(t)) \in \R^{2N}\] the states of the system consisting of $N$ pedestrians.
The model equations are of Newton-type dynamics
\begin{align}
\dot{x}_i(t) &= V(x_i(t),v_i(t)),\notag \\
\dot{v}_i(t) &= F^{dest}(x_i(t),v_i(t))+F_i^{int}(\vec{x}(t),\vec{v}(t)) \label{eq:DetNewtonian}
\end{align}
with initial positions $\vec{x}(0) = \vec{x}_0 \in \R^{2N}$ and velocities $\vec{v}(0) = \vec{v}_0 \in \R^{2N}$.
The acting forces are divided into the destination force $F^{dest}$ and the interaction force $F_i^{int}$. 
The determination of boundary and obstacles forces is done using the function $V$ adapted from swarming models \cite{Armbruster2016_1,Armbruster2016_2}. We remark that the equations~\eqref{eq:DetNewtonian} are non-standard in the sense
that the first equation on $\dot{x}_i(t)$ is not dependent on the velocity $v_i(t)$ only. The motivation of 
$V(x_i(t),v_i(t))$ will be explained more detailed in the paragraph on obstacle forces.

To keep the notation and calculations well-arranged, we keep the basic model simple and focus on the new modeling ideas.

\subsubsection*{Interaction and Destination Forces}
As in \cite{HelbingMolnar1998}, the interaction force acting on pedestrian $i$ is
 \[F_i^{int}(\vec{x}(t),\vec{v}(t)):= \frac{1}{N-1}\sum_{\substack{j=1\\ j\neq i}}^N G(x_i(t)-x_j(t)),\] where $G\colon \R\to \R^2$ is a given vector field describing the repulsion and attraction.
\\
Let $v^C>0$ be the comfort speed which is achieved approximately in the relaxation time $\tau >0$. We define the destination force by
\[F^{dest}(x_i(t),v_i(t)):= \frac{1}{\tau}\left(v^C D(x_i(t))-v_i(t)\right),\]
where $D \colon \R^2 \to \R^2$ describes the direction to the destination of a pedestrian at position $x_i(t)$. If $x^D \in \R^2$ is some destination point, then $D$ can be expressed in the simplest case by \[D(x) = \frac{x^D-x}{||x^D-x||}.\]
Note that alternative destination directions could be achieved by considering the shortest path between the current position $x$ and the destination $x^D$ by solving the Eikonal equation, see e.g. \cite{Etikyala2014}.

\subsubsection*{Obstacle Forces}
Different to \cite{HelbingMolnar1998}, we choose an alternative way to model the obstacle forces by applying
a kind of specular reflection, see \cite{Armbruster2016_1, Armbruster2016_2}. 
The reason for this choice is that singular forces can occur close to obstacles leading to
numerical difficulties during the simulation of~\eqref{eq:DetNewtonian}. 

Let $\Gamma \subset \R^2$ be a domain with outer normal unit field $\vec{n}$ which is continuous and defined on $\overline{\Gamma}$.
The boundary $\partial \Gamma$ describes walls as well as obstacles and the outer normal field allows to manipulate the velocity vector such that pedestrians walk along these boundaries.
We assume a pedestrian very close to a wall at position $x \in \Gamma$ having a velocity vector $v \in \R^2$ which is given by the destination and interaction forces. The position at a later time is then given by $x+\Delta t v$ and might be outside of the domain, see figure \ref{fig:VeloBound}. 

	\begin{figure}
	\subfloat{
		\begin{tikzpicture}[scale = 0.95]
		\begin{axis}[grid=none,
          xmax=1.5,ymax=1.5,xmin=-0.5,ymin=-0.5,
          axis lines=middle,
          restrict x to domain= 0:1]
		\addplot[black,domain=0:1,samples=100] {sqrt(1-x^2)};
		\addplot[black,dashed,domain=0:0.8,samples = 80] {sqrt(0.8^2-x^2};

%
		\draw[->,thick] (0.6,0.6)--(0.85,0.7);
		\draw (0.89,0.71) node{$v$};
		\filldraw[black] (0.6,0.6) circle (2pt) node[align=right,below]{$x$};

		\end{axis}
	\end{tikzpicture}
	}
	\subfloat{
	\begin{tikzpicture}[scale = 0.95]
		\begin{axis}[grid=none,
          xmax=1.5,ymax=1.5,xmin=-0.5,ymin=-0.5,
          axis lines=middle,
          restrict x to domain= 0:1]
		\addplot[black,domain=0:1,samples=100] {sqrt(1-x^2)};
		\addplot[black,dashed,domain=0:0.8,samples = 80] {sqrt(0.8^2-x^2};	
	    \draw[->,thick] (0.6,0.6)--(0.8412,0.4804);
	    \draw (0.92,0.48) node{$v^\ast$};
		\filldraw[black] (0.6,0.6) circle (2pt) node[align=right,below]{$x$};
				\draw[->,thick,black,dashed] (0.6,0.6)--(0.85,0.7);
				\draw (0.89,0.71) node{$v$};
				\addplot[black,dashed,domain=0:1,samples=100] {x};
				\addplot[black,dashed,domain=0.5:1,samples = 100] {sqrt(1/2)-(x-sqrt(1/2))};
				\draw[->,thick,black,dashed] (0.6,0.6)--(0.7904,0.4096);
				\draw (0.83,0.35) node{$\tilde{v}$};

		\end{axis}
	\end{tikzpicture}
	}
	\caption{Velocity vector at the boundary}
	\label{fig:VeloBound}
	\end{figure}
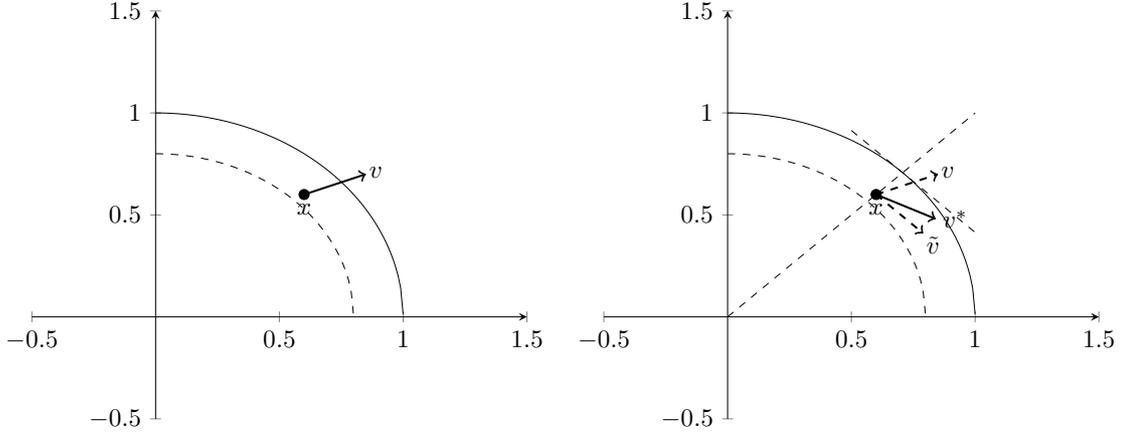	

This is the case if the vector $v$ points out of the domain, i.e.,\ $v\cdot \vec{n} \geq 0$. 
In the other case $v\cdot \vec{n} < 0$, the pedestrian walks into the domain and the velocity vector is admissible.
Let $d(x,\partial \Gamma)$ be the distance to the boundary and $U_\epsilon = \{x \in \R^2 \colon d(x,\partial \Gamma)\leq \epsilon\}$  the zone, where pedestrians feel uncomfortable close to a wall or obstacle.
We avoid the collision with the boundary by an orthogonal projection of the point $x+\Delta t v$ onto the tangent space at $x$ spanned by the orthogonal vector
\[\vec{n}^\perp (x) =(-n_2(x),n_1(x))^T\] 
to the outer unit normal vector. 
For the projected point, it follows
\[\Pi(x+\Delta t v) = x+ \Delta t v\cdot \vec{n}^\perp (x)\vec{n}^\perp (x)\]
and to conserve the velocity by the projection, we set 
\[\tilde{v}(v,x) = ||v||\frac{v\cdot \vec{n}^\perp(x)}{|v\cdot \vec{n}^\perp(x)|}\vec{n}^\perp(x) = ||v||\sgn{(v\cdot \vec{n}^\perp(x))}\vec{n}^\perp(x)\]
such that $||\tilde{v}(v,x)|| = ||v||$ holds.
Due to the assumption that a pedestrian will not change the direction immediately at the obstacle, we take a sufficiently smooth increasing function $J\colon [0,\infty)\to [0,1]$ with $J(0) = 0$ and $J(1) = 1$ and define
\[v^\ast(v,x) = \tilde{v}(v,x)+J\left(\frac{d(x,\partial \Gamma)}{\epsilon}\right)(v-\tilde{v}(v,x))\]
Summarizing, the velocity vector is given by
\[V(x,v):= 
\begin{cases}
v &\text{ if } v\cdot \vec{n}(x) < 0 \text{ and }\\
v^\ast(v,x) \frac{||v||}{||v^\ast(v,x)||}&\text{ if } v \cdot \vec{n}(x) \geq 0
\end{cases}
 \]
and one can easily check that $||V(x,v)|| = ||v||$. 
 
\subsubsection*{Random Velocity Effects}
All forces (interaction, destination and obstacle) considered so far are purely deterministic
and include no random effects. 
Different to pedestrian models modeled by SDEs, we choose an alternative way to include random effects. From phenomenological observations we deduce the intrinsic random effect that pedestrians can decide to walk or to stop independently of each other. 
This can be done by switching the velocity and acceleration to zero and back again. 
More precisely, an individual that stops immediately will keep on walking again after a random time.  

To include this phenomenon mathematically we describe the positions, velocities and the status, i.e.,\, stop or go, as a time discrete stochastic process. Let $\Delta t>0$ be a small fixed time stepsize and let $x_i^n \in \R^2$ the position, $v_i^n \in \R^2$ the velocity and $r_i^n \in \{0,1\}$ the status of pedestrian $i$ at time $t_n = n \Delta t$ for $n \in \N_0$. 

To ease the notation, we denote by $F_i = F^{dest}+F_i^{int}$ the sum of forces acting on pedestrian $i$. 

We consider a stochastic process $X = (X^n,n\in \N_0)$ with $X^n = (\vec{x}^n,\vec{v}^n,\vec{r}^n) \in E^{(N)} = \R^{2N}\times \R^{2N}\times \{0,1\}^N$ on some probability space $(\Omega,\mathcal{A},P)$ such that 
\begin{align}
P(x_i^{n+1} &= x_i^n+r_i^n\Delta t  V(x_i^n,v_i^n)|X^n) =1,\label{eq:x}\\
P(v_i^{n+1} &= v_i^n + \Delta t  F_i(\vec{x}^n,\vec{v}^n)|X^n)= r_i^n,\label{eq:v1}\\
P(v_i^{n+1} &= 0|X^n) =1-r_i^n,\label{eq:v2} \text{ and }\\
P(r_i^{n+1} &= z|X^n) = \Ind_z(r_i^n)(1-\Delta t \lambda(r_i^n,x_i^n))+\Ind_{1-z}(r_i^n)\Delta t \lambda(r_i^n,x_i^n) \label{eq:r}
\end{align}
for all $i = 1,\dots,N$, $n\in \N_0$, $z\in \{0,1\}$ and some rate function $\lambda \colon \{0,1\}\times \R^2 \to \R_{\geq 0}$.
The rate function $\lambda$ describes the expected events per unit time and is a given parameter. To avoid doubling effects by the interaction forces and the effects arising from the rate function $\lambda$, the rate function is assumed to be only dependent on the position of the pedestrian $i$ itself. 
Equations \eqref{eq:x}--\eqref{eq:v2} can be seen as an Euler approximation of the deterministic system \eqref{eq:DetNewtonian} while equation \eqref{eq:r} describes the probability to walk ($z=1$) or to stop ($z=0$) during the next time step $t_{n+1}$ 
given the values of the process at time $t_{n}$. 
This idea is adapted from \cite{DegondRinghofer2007}, where 
a continuous-time Markov Chains is used to motivate a production model with random breakdowns. 

Note that for the approach~\eqref{eq:x}--\eqref{eq:r}, 
we can only expect a well-defined model if $\Delta t \lambda(r,x) \in [0,1]$. Therefore we assume $\lambda$ to be uniformly bounded and $\Delta t \leq ||\lambda||_\infty^{-1}$. 
Furthermore, we have to fix an initial probability measure $\mu^0_{(N)}$ on the measurable space $(E^{(N)},\mathcal{E}^{(N)})$, where we denote by $\mathcal{E}^{(N)} = \sigma(E^{(N)})$ the smallest $\sigma$-algebra containing $E^{(N)}$.
Summarizing, we state the following definition.

\begin{definition}
A time discrete stochastic process $X$ on some probability space $(\Omega,\mathcal{A},P^{\mu_{(N)}^0})$ with values in $E^{(N)} = \R^{2N}\times \R^{2N}\times \{0,1\}^N$ and initial measure $\mu_{(N)}^0$ satisfying equations \eqref{eq:x}--\eqref{eq:r} is called a \emph{stochastic microscopic pedestrian model}.
\end{definition}

Now, we will show the well-posedness of the model. We define the mapping $U \colon E^{(N)}\times \mathcal{E}^{(N)}\to \R_{\geq 0}$ by 
\begin{align}
U((\vec{x},\vec{v},\vec{r}),B) := \sum_{\vec{z}\in \{0,1\}^N}\left[\epsilon_{h(\vec{x},\vec{v},\vec{r},\vec{z})}(B)\prod_{i=1}^N\left(\Ind_{z_i}(r_i)(1-\Delta t \lambda(r_i,x_i))+\Ind_{1-z_i}(r_i)\Delta t \lambda(r_i,x_i)\right)\right] \label{eq:U}
\end{align} 
with vector valued function 
\[h(\vec{x},\vec{v},\vec{r},\vec{z})=
\begin{pmatrix}
&x_1+\Delta t r_1 V(x_1,v_1)\\
&\vdots \\
&x_N+\Delta t r_N V(x_N,v_N)\\
&r_1(v_1+\Delta t F_1(\vec{x},\vec{v}))\\
&\vdots\\
&r_N(v_N+\Delta t F_N(\vec{x},\vec{v}))\\
&\vec{z}
\end{pmatrix}.
\]
Then, well-posedness can be obtained using the Markovian kernel property.

\begin{theorem}\label{thm:existence_micro}
Let $\lambda,V,F^{dest},F_i^{int}, i=1,\dots,N$ be measurable mappings, $\lambda$ uniformly bounded 
and $\Delta t \leq ||\lambda||_\infty^{-1}$. Then, the mapping $U$ defined by \eqref{eq:U} is a Markovian kernel on $(E^{(N)},\mathcal{E}^{(N)})$.
\end{theorem}

\begin{proof}
For simplicity, we identify $x$ as $(\vec{x},\vec{v},\vec{r})$ in the following. 
The first step of the proof is to show that 
for every set $B\in \mathcal{E}^{(N)}$ the mapping $x \mapsto U(x,B)$ is measurable. 
In a second step, we have to prove that for every $x \in E^{(N)}$ the mapping $B \mapsto U(x,B)$ is a probability measure on $(E^{(N)},\mathcal{E}^{(N)})$.
From the assumptions we get that $h$ is a measurable mapping and hence $x\mapsto \epsilon_{h(x)}(B)$ is measurable. 
The function $\lambda$ is also measurable and the sum as well as the product of \eqref{eq:U} are finite such that $U(x,B)$ is measurable in $x$.

The mapping $U$ consists of a finite sum of weighted Dirac measures on $(E^{(N)},\mathcal{E}^{(N)})$ and the assumption $0<\Delta t \leq ||\lambda||_\infty^{-1}$ implies that weights must be non-negative. 
So $U$ is a measure on $(E^{(N)},\mathcal{E}^{(N)})$ and it remains to show that $U(x,E^{(N)}) = 1$ to complete the proof. We calculate
\begin{align*}
U(x,E^{(N)}) &=  \sum_{z\in \{0,1\}^N}\left[\prod_{i=1}^N\left(\Ind_{z_i}(r_i)(1-\Delta t \lambda(r_i,x_i))+\Ind_{1-z_i}(r_i)\Delta t \lambda(r_i,x_i)\right)\right]\\
	&= \sum_{z_1=0}^1 \dots \sum_{z_{N-1}=0}^1 \prod_{i=1}^{N-1}\left(\Ind_{z_i}(r_i)(1-\Delta t \lambda(r_i,x_i))+\Ind_{1-z_i}(r_i)\Delta t \lambda(r_i,x_i)\right)\\
	&\; \cdot \sum_{z_N=0}^1\left(\Ind_{z_N}(r_N)(1-\Delta t \lambda(r_N,x_N))+\Ind_{1-z_N}(r_N)\Delta t \lambda(r_N,x_N)\right)\\
	&= \sum_{z_1=0}^1 \dots \sum_{z_{N-1}=0}^1 \prod_{i=1}^{N-1}\left(\Ind_{z_i}(r_i)(1-\Delta t \lambda(r_i,x_i))+\Ind_{1-z_i}(r_i)\Delta t \lambda(r_i,x_i)\right)\cdot 1\\
	&= \dots \\
	&= 1
\end{align*}
and conclude that $U$ is a probability measure in the second component.
\end{proof}

\noindent The mapping $U$ describes the one step transition probability of the system and allows to define a Markovian semigroup of kernels. We choose the typical composition of Markovian kernels, see e.g.\,\cite{BauerWTengl}, 
\begin{align}
U \circ U (x,B) &:= \int_E U(x,dy)U(y,B) \label{eq:Uop}
\end{align}
and define 
\begin{align}
U^0 &:= \operatorname{Id}\label{eq:UHG1}\\
U^n &:= \underset{k=1}{\overset{n}{\circ}}\; U \text{ for } n\in \N.\label{eq:UHG2}
\end{align}
The family $(U^n,n\in \N_0)$ is then a normal semigroup of Markovian kernels and we get the following
well-posedness result.

\begin{theorem}
Let $\lambda,V,F^{dest},F_i^{int}, i=1,\dots,N$ be measurable mappings, $\lambda$ uniformly bounded, $\Delta t \leq ||\lambda||_\infty^{-1}$ and $\mu^0_{(N)}$ a probability measure on $(E^{(N)},\mathcal{E}^{(N)})$. Then, there exists a \emph{stochastic microscopic pedestrian model}. Furthermore, the distribution of $N$ pedestrians at time step $t_n$ is given by
\begin{align}
\mu^n_{(N)}(B) = \int_{E^{(N)}} U^n(x,B)\mu^0_{(N)}(dx) \label{eq:mu^n}
\end{align}
for every $B \in \mathcal{E}^{(N)}$.
\end{theorem}  

\begin{proof}
The state space $E^{(N)}$ is a polish space and we hence can use the Daniell-Kolmogorov theorem~\cite{BauerWTengl} which guarantees the existence of a canonical coordinate process $X = (X^n,n\in \N_0)$ on some probability space $(\Omega,\mathcal{A},P^{\mu^0_{(N)}})$ defined by the semigroup of Markovian kernels $(U^n, n\in \N_0)$, cf. \eqref{eq:UHG2}. The stochastic process $X$ is a Markov process and 
therefore \eqref{eq:mu^n} and
\[P^{\mu^0_{(N)}}(X^{n+1} \in B|X^n = x) = U(x,B)\]
are satisfied.
By setting $x = (\vec{x},\vec{v},\vec{r})$ and 
\[B = \{x_1+r_1 \Delta t  V(x_1,v_1)\}\times \R^{2(N-1)} \times \R^{2N} \times \{0,1\}^N \]
 we obtain equation \eqref{eq:x} in the case $i=1$. 
The remaining equations \eqref{eq:v1}--\eqref{eq:r} can be deduced analogously.
\end{proof}

%% file: kinetic_model.tex
\subsection{Kinetic Model}
We now analyze the evolution of measures given by the microscopic pedestrian flow model. The principal idea is to derive the Kolmogorov forward equation, see e.g.\ \cite{GikhmanStochProc2}, given by the semigroup $(U_n,n\in \N_0)$. 
To do so, we introduce the following spaces of test-functions $\mathcal{C}^{(N)}$ and $\mathcal{C}$ to capture the discrete states of the pedestrians:
\begin{eqnarray*}
\mathcal{C}^{(N)}&:=&\{\phi \colon E^{(N)} \to \R \colon \phi(\vec{x},\vec{v},\cdot) \text{ is } \mathcal{P}(\{0,1\}^N)/\mathcal{B}(\R)\text{-measurable}, \phi(\cdot,\cdot,\vec{r}) \in C_b^1(\R^{2N}\times \R^{2N};\R)\} \\ 
\mathcal{C}&:=&\{\psi \colon E \to \R \colon \psi(x,v,\cdot) \text{ is } \mathcal{P}(\{0,1\})/\mathcal{B}(\R)\text{-measurable}, \psi(\cdot,\cdot,r) \in C_b^1(\R^2\times \R^2;\R)\},
\end{eqnarray*}
where $E = \R^2 \times \R^2 \times \{0,1\}$.
Here, $\mathcal{P}$ denotes the power set and $\mathcal{B}(\R)$ the Borel $\sigma$-algebra on $\R$.
We make the following assumption: 
For every $n \in \N_0$ there exists a probability measure $\mu^n_{(1)} $ on $(E,\mathcal{E})$ such that 
for all $B = B^x_1\times \cdots \times B^x_N\times B^v_1\times \cdots \times B^v_N\times B^r_1\times \cdots \times B^r_N \in \mathcal{E}^{(N)}$ with $B_i^x \times B_i^v \times B_i^r \in \mathcal{E}, i=1,\dots,N$, we have
\begin{align}
\mu^n_{(N)} (B) = \prod_{i=1}^N \mu^n_{(1)}(B^x_i \times B^v_i \times B^r_i).\label{eq:MolChaos}
\end{align}
This is the so-called molecular chaos assumption which forces the $N$-particle distribution being the product measure of a one particle distribution. In fact, for existing models it is shown that this is true as $N$ tends to infinity, see e.g.\ \cite{Chen2016}.
To keep the derivation of the kinetic equation clearly arranged, we state the following lemma. 
\begin{lemma}\label{lem:OneStep}\hspace{0mm}

\begin{enumerate}
\item We have the representation
\begin{align*}
&\prod_{i=1}^N\left(\Ind_{z_i}(r_i)(1-\Delta t \lambda(r_i,x_i))+\Ind_{1-z_i}(r_i)\Delta t \lambda(r_i,x_i)\right)\\
=\; &\prod_{i=1}^N \Ind_{z_i}(r_i) (1-\Delta t \sum_{j=1}^N \lambda(r_j,x_j)) +\Delta t\sum_{j=1}^N \lambda(r_j,x_j)\Ind_{1-z_j}(r_j)\prod_{i\neq j}\Ind_{z_i}(r_i) + o(\Delta t) 
\end{align*}
as $\Delta t \to 0$.
\item
For every $\phi \in \mathcal{C}^{(n)}$ it holds
\begin{align*}
U\phi(\vec{x},\vec{v},\vec{r}) = \phi(\vec{x},\vec{v}\vec{r},\vec{r})&+\Delta t \sum_{j=1}^N \lambda(r_j,x_j)(\phi(\vec{x},\vec{v}\vec{r},\theta_j (\vec{r}))-\phi(\vec{x},\vec{v}\vec{r},\vec{r}))\\
	&+\Delta t \sum_{j=1}^N r_j(\nabla_{x_j}\phi(\vec{x},\vec{v}\vec{r},\vec{r})\cdot V(x_j,v_j)+\nabla_{v_j}\phi(\vec{x},\vec{v}\vec{r},\vec{r})\cdot F_j(\vec{x},\vec{v}))\\
	&+ o(\Delta t),
\end{align*}
where we set $\vec{v}\vec{r} = (v_1 r_1,\dots,v_N r_N)$ and $\theta_j(\vec{r}) = (r_1,\dots,r_{j-1},1-r_j,r_{j+1},\dots,r_N)$.
\end{enumerate}
\end{lemma}
\begin{proof}
The first part of the lemma directly follows by rearranging and collecting the $o(\Delta t)$ terms. 
Thus, it remains to show the second part. We have
\begin{align}
U\phi(\vec{x},\vec{v},\vec{r}) &:= \int_{E^{(N)}} \phi(\vec{y},\vec{w},\vec{s}) U((\vec{x},\vec{v},\vec{r}),d(\vec{y},\vec{w},\vec{s}))\notag \\
	&= \sum_{\vec{z} \in \{0,1\}^N} \phi(h(\vec{x},\vec{v},\vec{r},\vec{z}))\prod_{i=1}^N\left(\Ind_{z_i}(r_i)(1-\Delta t \lambda(r_i,x_i))+\Ind_{1-z_i}(r_i)\Delta t \lambda(r_i,x_i)\right). \label{eq:251}
\end{align}
If we apply a Taylor expansion to $\phi(h(\vec{x},\vec{v},\vec{r},\vec{z}))$, this reads
\begin{align}
\phi(h(\vec{x},\vec{v},\vec{r},\vec{z})) =&\; \phi(\vec{x},\vec{v}\vec{r},\vec{z})\notag\\
&\;+ \Delta t \sum_{j=1}^N r_j(\nabla_{x_j}\phi(\vec{x},\vec{v}\vec{r},\vec{z})\cdot V(x_j,v_j)+\nabla_{v_j}\phi(\vec{x},\vec{v}\vec{r},\vec{z})\cdot F_j(\vec{x},\vec{v}))\notag\\
&\; + o(\Delta t).\label{eq:252}
\end{align}
Inserting \eqref{eq:252} into \eqref{eq:251} and using the first part of the lemma we end up with
\begin{align*}
U\phi(\vec{x},\vec{v},\vec{r}) =&\; \phi(\vec{x},\vec{v}\vec{r},\vec{r})\\
&\;-\Delta t \sum_{j=1}^N \lambda(r_j,x_j)\phi(\vec{x},\vec{v}\vec{r},\vec{r})\\
&\;+\Delta t \sum_{j=1}^N \lambda(r_j,x_j)\phi(\vec{x},\vec{v}\vec{r},\theta_j (\vec{r}))\\
&\;+\Delta t \sum_{j=1}^N r_j(\nabla_{x_j}\phi(\vec{x},\vec{v}\vec{r},\vec{r})\cdot V(x_j,v_j)+\nabla_{v_j}\phi(\vec{x},\vec{v}\vec{r},\vec{r})\cdot F_j(\vec{x},\vec{v}))\\
	&+ o(\Delta t).
\end{align*}
\end{proof}

Furthermore, to state the discrete time evolution of measures, we have to define a consistency relation.
\begin{definition}
The consistency relation for the initial measure $\mu_{(N)}^0$ is defined by
\begin{align}
\int_{E^{(N)}} \phi(\vec{x},\vec{v} \vec{r},\vec{r}) \mu_{(N)}^0(d(\vec{x},\vec{v},\vec{r})) = \int_{E^{(N)}} \phi(\vec{x},\vec{v},\vec{r}) \mu_{(N)}^0(d(\vec{x},\vec{v},\vec{r})). \label{eq:consN}
\end{align}
\end{definition}
The relation \eqref{eq:consN} means that non-moving pedestrians at the initial time $t_0$ have zero velocity.
The next lemma shows that the consistency relation is also conserved in time.
\begin{lemma}\label{lem:consN}
Let the measure $\mu_{(N)}^0$ satisfy the consistency relation \eqref{eq:consN}, then it holds
\begin{align}
\int_{E^{(N)}} \phi(\vec{x},\vec{v} \vec{r},\vec{r}) \mu_{(N)}^n(d(\vec{x},\vec{v},\vec{r})) = \int_{E^{(N)}} \phi(\vec{x},\vec{v},\vec{r}) \mu_{(N)}^n(d(\vec{x},\vec{v},\vec{r})). \label{eq:consN2}
\end{align}
for every $n \in \N_0$.
\end{lemma}
\begin{proof}
The case $n = 0$ is true  by assumption. Let us further assume that \eqref{eq:consN2} is true for some arbitrary, fixed $n \in \N_0$, 
then due to \eqref{eq:mu^n}
\begin{align*}
\int_{E^{(N)}} \phi(\vec{x},\vec{v},\vec{r}) \mu_{(N)}^{n+1}(d(\vec{x},\vec{v},\vec{r})) = 
\int_{E^{(N)}} U\phi(\vec{x},\vec{v},\vec{r}) \mu_{(N)}^n(d(\vec{x},\vec{v},\vec{r})).
\end{align*}
It remains to show that
\begin{align*}
U\phi(\vec{x},\vec{v},\vec{r}) = U\phi(\vec{x},\vec{v} \vec{r},\vec{r}).
\end{align*}
We know that
\begin{align*}
V(x_i,v_i r_i) &= r_i V(x_i,v_i)
\end{align*}
and hence
\[h((\vec{x},\vec{v} \vec{r},\vec{r})) = h((\vec{x},\vec{v},\vec{r}))\] due to $r_i \in \{0,1\},$ i.e., $r_i^2 = r_i$.
Then it is straightforward to see \[U((\vec{x},\vec{v} \vec{r},\vec{r}),B) =  U((\vec{x},\vec{v},\vec{r}),B)\] and 
\begin{align*}
U\phi(\vec{x},\vec{v},\vec{r}) = U\phi(\vec{x},\vec{v} \vec{r},\vec{r}).
\end{align*}
\end{proof}

The discrete time evolution of measures can be finally summarized in the following theorem.

\begin{theorem}\label{thm:N_part_dist}
Let the consistency relation \eqref{eq:consN} be satisfied. Then, for all $\phi \in \mathcal{C}^{(N)}$, the $N$-particle distribution $(\mu_{(N)}^n, n\in \N_0)$ satisfies
\begin{align}
&\frac{1}{\Delta t}\left(\int_{E^{(N)}} \phi(\vec{x},\vec{v},\vec{r}) \mu^{n+1}_{(N)}(d(\vec{x},\vec{v},\vec{r}))-\int_{E^{(N)}} \phi(\vec{x},\vec{v},\vec{r}) \mu^{n}_{(N)}(d(\vec{x},\vec{v},\vec{r}))\right)\notag \\
=& \int_{E^{(N)}}\sum_{j=1}^N \Big[r_j\nabla_{x_j} \phi(\vec{x},\vec{v}\vec{r},\vec{r}) \cdot V(x_j,v_j)+r_j \nabla_{v_j} \phi(\vec{x},\vec{v} \vec{r},\vec{r}) \cdot F_j(\vec{x},\vec{v})\notag \\
&+\lambda(r_j,x_j)(\phi(\vec{x},\vec{v}\vec{r},\theta_j(\vec{r}))-\phi(\vec{x},\vec{v} \vec{r},\vec{r}))\Big]\mu^n_{(N)}(d(\vec{x},\vec{v},\vec{r}))\notag\\
&+o(1)\notag.
\end{align}
\end{theorem}
\begin{proof}
The result is a direct application of lemma \ref{lem:OneStep} and \ref{lem:consN}.
\end{proof}

The kinetic model to the stochastic microscopic pedestrian flow model can be derived
as follows:
If we apply theorem \ref{thm:N_part_dist} on the function $\phi(\vec{x},\vec{v},\vec{r}) = \psi(x_1,v_1,r_1) \in \mathcal{C}^{(N)}$, we get 
\begin{align}
&\frac{1}{\Delta t}\left(\int_{E} \psi(x_1,v_1,r_1) \mu^{n+1}_{(1)}(d(x_1,v_1,r_1))-\int_{E} \psi(x_1,v_1,r_1) \mu^{n}_{(1)}(d(x_1,v_1,r_1))\right)\notag \\
=& \int_{E} \Big[r_1\nabla_{x_1} \psi(x_1,v_1 r_1,r_1) \cdot V(x_1,v_1)+r_1\nabla_{v_1} \psi(x_1,v_1 r_1,r_1) \cdot F^{dest}(x_1,v_1)\notag \\
&+\lambda(r_1,x_1)(\psi(x_1,v_1 r_1,1-r_1)-\psi(x_1,v_1 r_1,r_1))\Big]\mu^n_{(1)}(d(x_1,v_1,r_1))\notag \\
&+\frac{1}{N-1}\sum_{j=2}^N\int_{E^{(N)}}r_1\nabla_{v_1} \psi(x_1,v_1 r_1,r_1)\cdot G(x_1-x_j)\mu^n_{(N)}(d(\vec{x},\vec{v},\vec{r}))\label{eq:N_part_conv}\\
&+o(1)\notag,
\end{align}
where $\mu_{(1)}^n$ denotes the distribution of the first pedestrian.

Using the molecular chaos assumption \eqref{eq:MolChaos}, we can express \eqref{eq:N_part_conv} as
\begin{align*}
\;&\int_{E^{(N)}} r_1 \nabla_{v_1} \psi(x_1,v_1 r_1,r_1)\cdot G(x_1-x_j)\mu^n_{(N)}(d(\vec{x},\vec{v},\vec{r}))\\
=&\; \int_{E} r_1 \nabla_{v_1} \psi(x_1,v_1 r_1,r_1)\cdot \int_E G(x_1-x_2)\mu^n_{(1)}(d(x_2,v_2,r_2)) \mu^n_{(1)}(d(x_1,v_1,r_1))
\end{align*}
and therefore
\begin{align*}
&\; \frac{1}{N-1}\sum_{j=2}^N\int_{E^{(N)}} r_1 \nabla_{v_1} \psi(x_1,v_1 r_1,r_1)\cdot G(x_1-x_j)\mu^n_{(N)}(d(\vec{x},\vec{v},\vec{r}))\\
= &\; \int_{E} r_1\nabla_{v_1} \psi(x_1,v_1 r_1,r_1)\cdot \int_E  G(x_1-x_2)\mu^n_{(1)}(d(x_2,v_2,r_2)) \mu^n_{(1)}(d(x_1,v_1,r_1)).
\end{align*}

We introduce $(\mu^t, t\geq 0)$ such that $\mu^n_{(1)} = \mu^{n\Delta t}$ and obtain a continuous time equation for $\mu^t$ by $\Delta t \to 0$ reading
\begin{align}
&\frac{d}{dt} \int_{E} \psi(x,v,r) \mu^{t}(d(x,v,r))\notag \\
=& \int_{E} \Big[r\nabla_{x} \psi(x,rv,r) \cdot V(x,v)+r\nabla_{v} \psi(x,rv,r) \cdot F^{dest}(x,v)\notag \\
&+\lambda(r,x)(\psi(x,rv,1-r)-\psi(x,rv,r))\notag\\
&+r \nabla_v \psi(x,rv,r)\cdot \int_{E} G(x-y)\mu^t(d(y,w,s))\Big]\mu^t(d(x,v,r)).\label{eq:kinetic_model}
\end{align} 

\begin{remark}
We prefer to use integral equations instead of a differential representation of~\eqref{eq:kinetic_model}. 
This is due to the fact that non-moving pedestrians have zero velocity and hence the measure
\[B^v \mapsto \mu^t(B^x \times B^v \times \{0\})\]
is a Dirac distribution ($\sim \epsilon_0(B^v)$). 
\end{remark}

Note that the one-particle distribution $\mu^t$ is often called kinetic model as the following definition indicates:
\begin{definition}
We call a family of measures $(\mu^t,t\geq 0)$ satisfying \eqref{eq:kinetic_model} for every $\psi \in \mathcal{C}$ \emph{stochastic kinetic pedestrian model}.
\end{definition}

%% file: macroscopic_model.tex
\subsection{Macroscopic Model}
As we have seen, 
the kinetic model describes the one particle distribution including detailed information about the velocity
of the pedestrians. However, we are mainly interested in the derivation of a simplified deterministic macroscopic model capturing
the stochastic dynamics and hence significant less computational costs.

Let \[\rho^t_k(B^x) := \int_{B^x\times \R^2 \times \{k\}} \mu^t(d(x,v,r))\]
be the probability to find a pedestrian in $B^x \in \mathcal{B}(\R^2)$ with status $k\in \{0,1\}$ and $\rho^t(B^x):= \rho^t_0(B^x)+\rho^t_1(B^x)$. If we choose $\phi(x,v,r) = \eta(x,r) \in \mathcal{C}$, then 
\[\frac{d}{dt}\int_E \phi(x,v,r) \mu^t(d(x,v,r)) = \frac{d}{dt} \int_{\R^2} \eta(x,0) \rho^t_0(dx)+\frac{d}{dt} \int_{\R^2} \eta(x,1) \rho^t_1(dx) \] 
and together with \eqref{eq:kinetic_model} we obtain
\begin{align*}
&\;\frac{d}{dt} \int_{\R^2} \eta(x,0) \rho^t_0(dx)+\frac{d}{dt} \int_{\R^2} \eta(x,1) \rho^t_1(dx)\\
=&\; \int_{E} r \nabla_x \eta(x,r) \cdot V(x,v) \mu^t(d(x,v,r))\\
&+\;\int_{\R^2} \lambda(0,x)(\eta(x,1)-\eta(x,0)) \rho_0^t(dx)+\int_{\R^2} \lambda(1,x)(\eta(x,0)-\eta(x,1)) \rho_1^t(dx),
\end{align*} 
where we consider no change in the velocity any more.

In a next step, we need to approximate the expression
\[\int_{E} r \nabla_x \eta(x,r) \cdot V(x,v) \mu^t(d(x,v,r))\]
with measures $\rho_0^t$ and $\rho_1^t$
to get a closed representation.
This can be achieved by the following rescaling arguments of the kinetic pedestrian flow model:
Let $t^\delta = t \delta$ and $x^\delta = x \delta$ be the rescaled variables satisfying $1\gg\delta>0$, then we can rewrite \eqref{eq:kinetic_model} as 
\begin{align}
&\frac{d}{dt} \int_{E} \psi(x,v,r) \mu^{t,\delta}(d(x,v,r))\notag \\
=& \int_{E} \Big[r \nabla_{x} \psi(x,rv,r) \cdot V(x,v)\notag \\
&+\frac{1}{\delta}[r\nabla_{v} \psi(x,rv,r) \cdot (F^{dest}(x,v)+\int_{\R^2}G(x-y)\rho^{t,\delta}(dy))\notag \\
&+\lambda(r,x)(\psi(x,rv,1-r)-\psi(x,rv,r))] \Big]\mu^{t,\delta}(d(x,v,r))\label{eq:kinetic_model_scaled}
\end{align} 
where the index $\delta$ at $x$ and $t$ is neglected.
By defining
\[
\textbf{F}(x,\rho^{t,\delta}) := \frac{v^C}{\tau}D(x) + \int_{\R^2}  G(x-y) \rho^{t,\delta}(dy),
\]
we get the new expression
\begin{align}
&\frac{1}{\delta}\int_E[r\nabla_{v} \psi(x,rv,r) \cdot (F^{dest}(x,v)+\int_{\R^2}G(x-y)\rho^{t,\delta}(dy))\notag\\
&\; +\lambda(r,x)(\psi(x,rv,1-r)-\psi(x,rv,r))] \Big]\mu^{t,\delta}(d(x,v,r))\notag \\
=&\;\frac{1}{\delta}\int_E\Big[r\nabla_{v} \psi(x,rv,r) \cdot \left(\textbf{F}(x,\rho^{t,\delta})-\frac{v}{\tau}\right)+\lambda(r,x)(\psi(x,rv,1-r)-\psi(x,rv,r))] \Big]\mu^{t,\delta}(d(x,v,r)) \label{eq:closure1}.
\end{align}

We remark that at this point it is not obvious whether a measure $\mu^{t,\delta}$ in the limit $\delta \to 0$ exists 
for equation \eqref{eq:closure1}. However,
in \cite{Jabin2000} it is analytically shown that for expressions 
of type  
\[\frac{1}{\delta}\int_E\Big[\nabla_{v} \psi(x,v) \cdot \left(\textbf{F}(x,\rho^{t,\delta})-\frac{v}{\tau}\right)\Big]\mu^{t,\delta}(d(x,v))\] 
if $G, F^{dest} \in W^{1,\infty}$, the sequence of measures $\mu^{t,\delta}$ weakly converges to a measure $\mu^{t,0}$ satisfying
\begin{align}
\mu^{t,0}(B^x \times B^v) = \int_{B^x}\epsilon_{\tau \textbf{F}(x,\rho^{t,0})}(B^v) \rho^{t,0}(dx) \label{eq:closure_orig}
\end{align}
for every $B^x,B^v \in \mathcal{B}(\R^2)$. 
This    
motivates the use of a Dirac distribution in the velocity to deduce a simplified macroscopic model. We set
\begin{align}
\mu^{t,\delta}_0(B^x\times B^v) &:= \epsilon_0(B^v) \int_{B^x} \frac{\lambda(1,x)}{\lambda(0,x)}\rho_1^{t,\delta}(dx) \text{ and } \label{eq:closure_dist1} \\
\mu^{t,\delta}_1(B^x\times B^v) &:= \int_{B^x} \epsilon_{v^\ast(x)}(B^v) \rho_1^{t,\delta}(dx) \label{eq:closure_dist2}
\end{align}
and intend to find a macroscopic velocity $v^\ast$ as closure distribution in the limit $\delta \to 0$. Inserting the measures \eqref{eq:closure_dist1} and \eqref{eq:closure_dist2} into \eqref{eq:closure1} yields
\begin{align}
&\;\int_E\Big[r\nabla_{v} \psi(x,rv,r) \cdot \left(\textbf{F}(x,\rho^{t,\delta})-\frac{v}{\tau}\right)+\lambda(r,x)(\psi(x,rv,1-r)-\psi(x,rv,r))] \Big]\mu^{t,\delta}(d(x,v,r))\notag \notag \\
=&\; \int_{\R^2} \lambda(1,x)(\psi(x,0,1)-\psi(x,0,0)) \rho_1^{t,\delta}(dx)\notag \\
&+\; \int_{\R^2} \lambda(1,x)(\psi(x,v^\ast(x),0)-\psi(x,v^\ast(x),1)) \rho_1^{t,\delta}(dx)\notag \\
&+ \; \int_{\R^2} \nabla_{v} \psi(x,v^\ast(x),1) \cdot \left(\textbf{F}(x,\rho^{t,\delta})-\frac{v^\ast(x)}{\tau}\right) \rho_1^{t,\delta}(dx). \label{eq:clos1}
\end{align}
We use a Taylor expansion to get
\[\psi(x,0,1)-\psi(x,v^\ast(x),1) = -\nabla_v \psi(x,v^\ast (x),1)v^\ast(x)+\frac{1}{2}v^\ast(x)^T \nabla_v^2 \psi(x,\theta_1 v^\ast(x),1)v^\ast(x)\] and 
\[\psi(x,v^\ast(x),0)-\psi(x,0,0) = \nabla_v \psi(x,\theta_2 v^\ast (x),0)v^\ast(x)\]
for some $\theta_1,\theta_2 \in [0,1]$.
and assuming that $\psi$ is twice differentiable with respect to $v$. 

Inserting the Taylor expansion into \eqref{eq:clos1} leads to
\begin{align}
\int_{\R^2} \nabla_{v} \psi(x,v^\ast(x),1) \cdot \left(\textbf{F}(x,\rho^{t,\delta})-\frac{v^\ast(x)}{\tau}-\lambda(1,x) v^\ast(x)\right) \rho_1^{t,\delta}(dx) + R \label{eq:clos2} 
\end{align}
with 
\begin{align*}
R = \int_{\R^2} [\nabla_v \psi(x,\theta_2 v^\ast(x),0)v^\ast(x)+\frac{1}{2}v^\ast(x)^T \nabla_v^2 \psi(x,\theta_1 v^\ast(x),1))v^\ast(x)]\lambda(1,x) \rho_1^t(dx).
\end{align*}

If we choose 
\begin{align}
v^\ast(x) = \frac{\tau \textbf{F}(x,\rho^{t,\delta})}{1+\tau \lambda(1,x)}\label{eq:closVelo}
\end{align}
as $\delta \to 0$,
the first term in \eqref{eq:clos2} vanishes and only $R$ remains. This choice for a closure distribution is close to the monokinetic closure \eqref{eq:closure_orig} but with scaling factor $(1+\tau \lambda(1,x))^{-1}\approx 1-\tau \lambda(1,x)$. 
So we end up with a meaningful result since 
the factor $-\tau \lambda(1,x)$ reduces the velocity about the expected number of stops during the relaxation time $\tau$. 

Finally, we collect all ideas and computations so far.
\begin{definition}\label{def:macro_model}
A family of measures $((\rho_0^t,\rho_1^t),t\geq 0)$ satisfying 
\begin{align}
&\;\frac{d}{dt} \int_{\R^2} \eta(x,0) \rho_0^t(dx) + \frac{d}{dt} \int_{\R^2} \eta(x,1) \rho_1^t(dx)\notag \\ =& \;\int_{\R^2} \nabla_x \eta(x,1) \cdot V\left(x,\frac{\tau \mathbf{F}(x,\rho^t)}{1+\tau \lambda(1,x)}\right)\rho_1^t(dx)\notag \\
& \;+ \int_{\R^2} \lambda(0,x) (\eta(x,1)-\eta(x,0)) \rho_0^t(dx) + \int_{\R^2} \lambda(1,x) (\eta(x,0)-\eta(x,1)) \rho_1^t(dx)\notag 
\end{align}
for every $\eta(\cdot,0), \eta(\cdot,1) \in C_b^1(\R^2;\R)$ with
\begin{align*}
\rho_0^0 = \rho_0^{I}, \qquad \rho_1^0 = \rho_1^{I}, \qquad \rho_0^{I}(\R^2)+\rho_1^{I}(\R^2) = 1
\end{align*}
for some positive measures $\rho_0^{I}, \rho_1^{I}$ on $(\R^2,\mathcal{B}(\R^2))$ is called a \emph{stochastic macroscopic pedestrian model}.
\end{definition}

We are able to obtain the differential representation of the stochastic macroscopic pedestrian model
by applying the closure distribution~\eqref{eq:closVelo} to eliminate the dynamics of the velocity.
To do so, we assume that the measures $\rho_0^t$ and $\rho_1^t$ are absolutely continuous with respect to the Lebesgue measure and possess density functions $u_0$ and $u_1$, respectively. Then,
\begin{align*}
&\;\frac{d}{dt} \int_{\R^2} \eta(x,0) u_0(t,x) dx + \frac{d}{dt} \int_{\R^2} \eta(x,1) u_1(t,x) dx\\
=& \;\int_{\R^2} \nabla_x \eta(x,1) \cdot V\left(x,\frac{\tau \mathbf{F}(x,\rho^t)}{1+\tau \lambda(1,x)}\right)u_1(t,x)dx \\
& \;+ \int_{\R^2} \lambda(0,x) (\eta(x,1)-\eta(x,0)) u_0(t,x)dx + \int_{\R^2} \lambda(1,x) (\eta(x,0)-\eta(x,1)) u_1(t,x)dx\\
=& \; \int_{\R^2} \eta(x,0) (\lambda(1,x) u_1(t,x)-\lambda(0,x) u_0(t,x))dx\\
& \;+\int_{\R^2} \nabla_x \eta(x,1) \cdot V\left(x,\frac{\tau \mathbf{F}(x,\rho^t)}{1+\tau \lambda(1,x)}\right)u_1(t,x) + \eta(x,1) (\lambda(0,x)u_0(t,x)-\lambda(1,x) u_1(t,x))dx 
\end{align*}
for every $\eta(\cdot,0),\eta(\cdot,1) \in C^1_b$.
Due to $C^1_b \subset C_0^\infty$, the macroscopic model implies the weak formulation  
\begin{align}
\partial_t u_0(t,x) &= \lambda(1,x)u_1(t,x)-\lambda(0,x)u_0(t,x), \label{eq:Macro21}\\
\partial_t u_1(t,x) &= \lambda(0,x)u_0(t,x)-\lambda(1,x)u_1(t,x) - \Div_x\left(V\left(x,\frac{\tau \mathbf{F}(x,u(t,\cdot))}{1+\tau \lambda(1,x)}\right)u_1(t,x)\right),\label{eq:Macro22}\\
u(t,x) &= u_0(t,x)+u_1(t,x) \label{eq:Macro23}
\end{align}
equipped with initial conditions $u_0(0,x) = g_0(x)\geq 0$, $u_1(0,x) = g_1(x) \geq 0$ satisfying \[\int_{\R^2} g_0(x)+g_1(x)dx = 1.\] 

Note that the latter representation of the stochastic macroscopic pedestrian model also allows for comparisons to the deterministic model setting. Figure \ref{fig:Overview} illustrates our main results starting from the microscopic model to the resulting first order macroscopic model using the proposed closure velocities.
The corresponding deterministic results can be directly obtained from \cite{Chen2016,Etikyala2014,Jabin2000}.

\begin{figure}[htb!]
\centering
\begin{tikzpicture}[font=\footnotesize,every node/.style={node distance=2cm},%
    force1/.style={draw, fill=black!0,inner sep=2mm,text width=0.1\textwidth, align=flush center,%
    minimum height=1.2cm, font=\bfseries\footnotesize},
    force2/.style={draw, fill=black!0,inner sep=2mm,text width=0.3\textwidth, align=flush center,%
    minimum height=1.2cm, font=\bfseries\footnotesize},
    force3/.style={draw, fill=black!0,inner sep=2mm,text width=0.45\textwidth, align=flush center,%
    minimum height=1.2cm, font=\bfseries\footnotesize},
    forceStoch/.style={draw, fill=black!0,inner sep=2mm,text width=0.45\textwidth, align=flush center,%
    minimum height=2cm, font=\bfseries\footnotesize},
    forceDet/.style={draw, fill=black!0,inner sep=2mm,text width=0.3\textwidth, align=flush center,%
    minimum height=2cm, font=\bfseries\footnotesize}]
    
  \node [force1](cc) {Closure Velocity};
  \node [force1,above of=cc](mic){Micro{\-}scopic};
  \node [force1,below of=cc](mac){Macro{\-}scopic};
  
  \node [forceDet,right=.5cm of mic](detmic){\vspace{-4mm}
  \begin{align*}
  \dot{x}_i &= V(x_i,v_i),\\
  \dot{v}_i & = F_i(\vec{x},\vec{v}).
  \end{align*}};
  \node [forceStoch,right=.5cm of detmic](stochmic){\vspace{-2.5mm}
  \begin{align*}
  P(x_i^{n+1} &= x_i^n+\Delta t r_i^n V(x_i^n,v_i^n)|X^n) = 1,\\
  P(v_i^{n+1} &= v_i^n+\Delta t F_i(\vec{x}^n,\vec{v}^n)|X^n) = r_i^n,\\
  P(r_i^{n+1} &= 1-r_i^n|X^n) = \Delta t \lambda(r_i^n,x_i^n).
  \end{align*}};
  
  \node [forceDet,right=.5cm of cc](detclos){\vspace{-3mm}
  \begin{align*}
	v^\ast = \tau \mathbf{F}(x,u)
  \end{align*}};
  \node [forceStoch,right=.5cm of detclos](stochclos){\vspace{-1.5mm}
  \begin{align*}
  	v^\ast_0 &= 0\\
	v^\ast_1 &= \frac{\tau}{1+\tau \lambda(1,x)} \mathbf{F}(x,u_0+u_1)
  \end{align*}};
  
  \node [forceDet,right=.5cm of mac](detmac){\vspace{-3mm}
  \begin{align*}
	\partial_t u = -\Div_x(V(x,v^\ast)u)
  \end{align*}};
  \node [forceStoch,right=.5cm of detmac](stochmac){\vspace{-2mm}
  \begin{align*}
	\partial_t u_0 &= \lambda_1 u_1-\lambda_0 u_0\\
	\partial_t u_1 &=\lambda_0 u_0-\lambda_1 u_1 -\Div_x(V(x,v^\ast_1)u_1)
  \end{align*}};
  
  \node [force2, above of=detmic](det){Deterministic};
  \node [force3, above of=stochmic](stoch){Stochastic};
  
  
  \path[-]
;
\end{tikzpicture}
\caption{Overview of the deterministic and stochastic model hierarchy equations}
\label{fig:Overview}
\end{figure}

%% file: numerics.tex
\section{Numerical Approximation}
\label{sec:3}
This section deals with the numerical treatment of the presented stochastic microscopic and macroscopic pedestrian model. 
First, we present a stochastic simulation algorithm for the microscopic model
to generate efficiently sample paths. Second, we introduce a numerical approximation scheme for the macroscopic model in its differential form.

\input{numerics_micro}

\input{numerics_macro}

%% file: numerics_micro.tex
\subsection{Microscopic Model}
The generation of sample paths for the microscopic model can be directly implemented with the transition kernel and the Markov property. In detail, let $X^n = (\vec{x}^n, \vec{v}^n,\vec{r}^n) \in E^{(N)}$ be the current state of the system and $ \Delta t \leq ||\lambda||_\infty^{-1}$. 
Due to \[P^{\mu^0_{(N)}}(x_i^{n+1} = x_i^n+r_i^n\Delta t V_i(\vec{x}^n,\vec{v}^n)|X^n) = 1\] we compute $x_i^{n+1} = x_i^n+r_i^n \Delta t V(x_i^n,v_i^n)$ and set $v_i^{n+1} = 0$ if $r_i^n = 0$ and $v_i^{n+1} = v_i^n+\Delta t F_i(\vec{x}^n,\vec{v}^n)$ if $r_i^n = 1$.
The equation
\[P^{\mu^0}(\vec{r}^{n+1} = z|X^n) = \prod_{i=1}^N\left(\Ind_{z_i}(r_i^n)(1-\Delta t \lambda(r_i^n,x_i^n))+\Ind_{1-z_i}(r_i^n)\Delta t \lambda(r_i^n,x_i^n)\right) \]  
implies that the status of each pedestrian $r_i^n$ to walk or to stop can be chosen independently according to a Bernoulli distribution with probabilities $p_i = \Delta t \lambda(r_i^n,x_i^n)$ where $\lambda$ is the given rate function.
Summarizing, we obtain algorithm \ref{algo:One_Step}.
\begin{algorithm}[htb!]
\caption{Euler approximation of the microscopic model for one time-step}
\label{algo:One_Step}
\KwData{Current state $(\vec{x}^n, \vec{v}^n,\vec{r}^n)$}
\Begin{
\For{$i =1:N$}{
	$x_i^{n+1} =  x_i^n+r_i^n \Delta t V(x_i^n,v_i^n)$\;
	\eIf{$r_i^n = 0$}{
	$v_i^{n+1} = 0$\;}
	{
	$v_i^{n+1} = v_i^n+\Delta t F_i(\vec{x}^n,\vec{v}^n)$\;
	}
	Generate $U \sim \mathcal{U}([0,1])$\;
	\eIf{$U \leq \Delta t \lambda(r_i^n,x_i^n)$}{
	$r_i^{n+1} = 1-r_i^n$\;	
	}{
	$r_i^{n+1} = r_i^n$\;
	}
}
}
\end{algorithm}

Note that we need to sample initial values according to the initial distribution $\mu^0$. Then, trajectories of the microscopic pedestrian model can be computed by algorithm~\ref{algo:One_Step} until a finite time horizon is reached.

%% file: numerics_macro.tex
\subsection{Macroscopic Model}
We consider equations \eqref{eq:Macro21}--\eqref{eq:Macro23}
to solve the macroscopic pedestrian model numerically by established approximation schemes.
Therefore, we decouple the original model into two separate problems (advection and reaction)
using a fractional-step method, see e.g. \cite{LevequeRed}, and solve 
the problems as follows:
\begin{enumerate}
\item Advection term \[\partial_t u_1+\Div_x\left(V\left(x,\frac{\tau \mathbf{F}(x,u(t,\cdot))}{1+\tau \lambda(1,x)}\right)u_1(t,x)\right) = 0. \]
This is a two dimensional scalar conservation law with space dependent, non-local flux function
that is solved by a first order Roe method while the convolution in \[\mathbf{F}(x,u(t,\cdot)) = \frac{v^C}{\tau}D(x)+\int_{\R^2}G(x-y)u(t,y)dy\] is approximated by a rectangular rule.
Dimension splitting \cite{LevequeRed} enables the numerical solution of one dimensional problems to solve the two dimensional advection equation in each time step. 
\item Reaction term 
\begin{align*}
\partial_t u_0(t,x) &= \lambda(1,x)u_1(t,x)-\lambda(0,x)u_0(t,x),\\
\partial_t u_1(t,x) &= \lambda(0,x)u_0(t,x)-\lambda(1,x)u_1(t,x).
\end{align*}
\end{enumerate} 
These equations can be explicitly solved for every $x \in \R^2$ since the reaction is linear.
Let \[u^R(t,x) = (u^R_0(t,x),u^R_1(t,x))\] denote the solution of the reaction equation, then 
\begin{align*}
u^R(t,x) &= e^{t \Lambda(x)}u^R(0,x)\\
&= \frac{1}{\overline{\lambda}(x)}
\begin{pmatrix}
(\lambda(1,x)+\lambda(0,x) e^{-t\overline{\lambda}(x)})u_0^R(0,x) &+ \lambda(1,x)(1-e^{-t\overline{\lambda}(x)})u_1^R(0,x)\\
\lambda(0,x)(1-e^{-t\overline{\lambda}(x)})u_0^R(0,x) &+
(\lambda(0,x)+\lambda(1,x) e^{-t\overline{\lambda}(x)})u_1^R(0,x)
\end{pmatrix}
\end{align*}
with $\overline{\lambda}(x) = \lambda(0,x)+\lambda(1,x)$ and 
\[\Lambda(x) = 
\begin{pmatrix}
-\lambda(0,x) & \lambda(1,x) \\
\lambda(0,x)  & -\lambda(1,x)
\end{pmatrix}.\]

%% file: comp_results.tex
\section{Computational Results}
\label{sec:4}
For simulation purposes,
we focus on new effects arising from the stochastic dynamics
and qualitative comparisons of the model hierarchy. 
We need to define performance measures to evaluate the quality of approximations of
the microscopic and macroscopic model. 
Therefore, we assume that the macroscopic measure $\rho^t$ admits a Lebesgue density $u(x,t)$ which is approximated by cell means $u_{ij}^{\text{Mac},n}$ on rectangles $Q_{ij}$ with lengths $\Delta x$, $\Delta y$ and center $x_{ij}$ at time step $t_n$. 

Furthermore, the average density of pedestrians in $Q_{ij}$ for the microscopic model is given by
\[u^{\text{Mic},n}_{ij}:=\frac{\mathbb{E}^{\mu^0_{(N)}}[\frac{1}{N}\sum_{l=1}^N\Ind_{Q_{ij}}(x_l^n)]}{\Delta x \Delta y}.\]
This allows to specify the error between the microscopic and macroscopic model as
\begin{align*}
\epsilon^n := u^{\text{Mic},n}-u^{\text{Mac},n} 
\end{align*}
measured in terms of the $L^p$-norm
\begin{align}
||\epsilon^n||_{L^p}&:=\left( \Delta x \Delta y \sum_{ij} |u^{\text{Mic},n}_{ij}-u^{\text{Mac},n}_{ij}|^p\right)^{\frac{1}{p}}
\end{align}
for every $p \geq 1$. 

The mass balance 
\begin{align}
\operatorname{MB}_{\hat{x}}(u(t,\cdot)) := \int_{(-\infty,\hat{x}]\times \R} u(t,x)dx \label{eq:MB} 
\end{align}
at some vertical cut at position $\hat{x} \in \R$ determines the amount of mass in front of the cut. 
This measure extracts information about the time needed that a certain percentage of pedestrians crossed the cut. 

In the following examples, we assume that the destination force is a point force \[D(x) = \frac{x^D-x}{||x^D-x||}\] with destination point $x^D \in \R^2$ and interaction kernel $G$ as the gradient of a Morse Potential, i.e.
\[G(x) = -2(e^{-(||x||-0.9)}-e^{-2(||x||-0.9)})\frac{x}{||x||}.\]

The microscopic model is only comparable to the macroscopic one if we assume that the initial velocities are given by the closure velocities of the macroscopic model, i.e.,\ for randomly generated positions $(x_1,\dots,x_N)$ and states $(r_1,\dots,r_N)$, we set
\[v_i = r_i\frac{\tau}{1+\tau \lambda(1,x_i)}\left (\frac{v^C}{\tau}D(x_i)+\frac{1}{N}\sum_{j=1}^NG(x_i-x_j)\right ),\] 
where the initial states $r_i$ are computed according to a Bernoulli distribution with initial probability $P(r_i = 0) = p_0 \in [0,1]$.

\subsubsection*{Example 1: Corridor with Space Dependent Rates}
As a first example, we consider a strip without obstacles with boundaries 
given by bold lines, see figure \ref{fig:Ex1DensSol}.
The comfort velocity $v^C$ and the relaxation time $\tau$ are assumed to be one.
The distribution of the initial positions should be
\[\rho^0(B) = \frac{1}{2}\int_B \Ind_{[-2,-1]\times[-1,1]}(x,y) d(x,y).\]
Additionally, the initial conditions for the macroscopic model are 
\[g_0(x,y) = p_0 \Ind_{[-2,-1]\times[-1,1]}(x,y) \text{ and } g_1(x,y) = (1-p_0)\Ind_{[-2,-1]\times[-1,1]}(x,y)\] supplemented by the initial probability to stop $p_0 = \frac{1}{2}$.

We are interested in the dynamic behavior of the model
for a space dependent rate function: 
\begin{align}
\lambda(0,x) &= 
\begin{cases}
	6 & \text{ if } ||x||\leq 0.5 \\
	10 &\text{ else,}
\end{cases} \label{eq:Ex1Rate1}
\\
\lambda(1,x) &= 
\begin{cases}
	5 & \text{ if } ||x||\leq 0.5 \\
	4 &\text{ else.}
\end{cases} \label{eq:Ex1Rate2}
\end{align}

The results of a simulation with destination $x^D = (100,0)$, $\Delta x = \Delta y = \frac{1}{40}$, $N = 100$ pedestrians and $M = 1000$ Monte Carlo iterations are shown in figure \ref{fig:Ex1DensSol}. 
For our computations, we have used a WIN10 operated PC with 4-core i5-4690 CPU 3.50GHz, 16GB DDR3 RAM.
All codes are implemented in MATLAB R2016b. The computation times we observed are  
5149 seconds for the microscopic and 434 seconds for the macroscopic model. 

In this scenario, the rate $\lambda$ describes the number of events per time such that equations \eqref{eq:Ex1Rate1} and \eqref{eq:Ex1Rate2} indicate that pedestrians will stop more frequently than intend to walk inside the circle around the origin with radius $0.5$.
We expect that walking pedestrians circulate around the non-moving crowd
and merge behind the bulk until a comfort density is reached again, cf. times $t = 10$ and $t = 15$ in figure \ref{fig:Ex1DensSol}.
As we can also see, the influence of the space dependent rate on the density is significant since the density is much higher in the region where people stop longer and walk shorter. 

\begin{figure}[htb!]
\centering
\subfloat
{
\includegraphics[width = 1\textwidth]{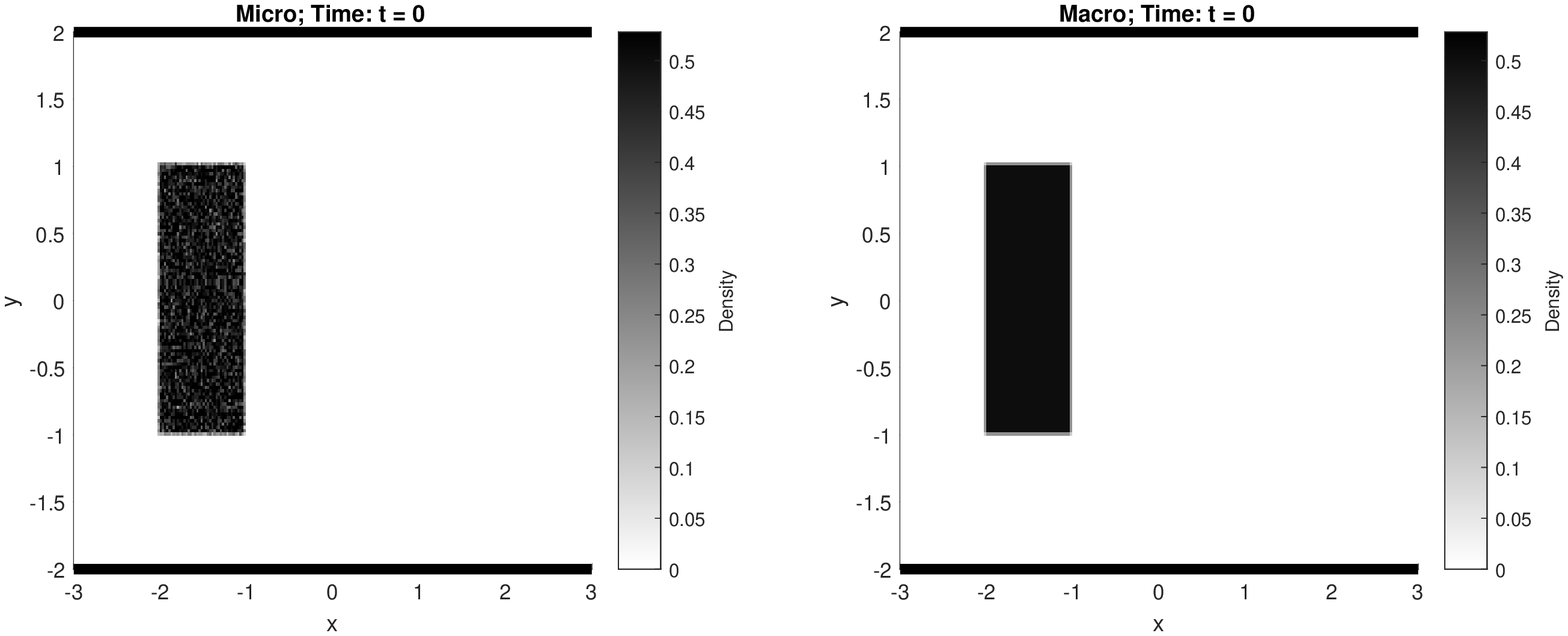}
}\qquad
\subfloat
{
\includegraphics[width = 1\textwidth]{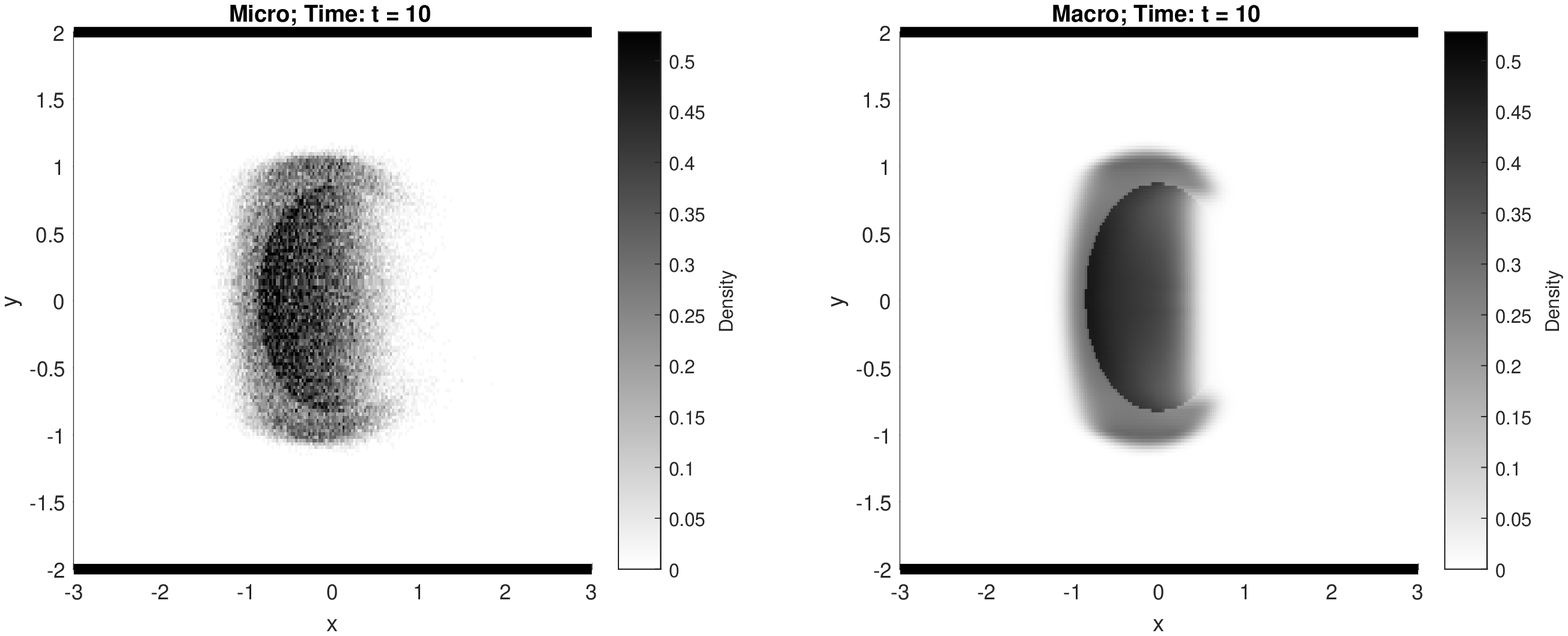}
}\qquad
\subfloat
{
\includegraphics[width = 1\textwidth]{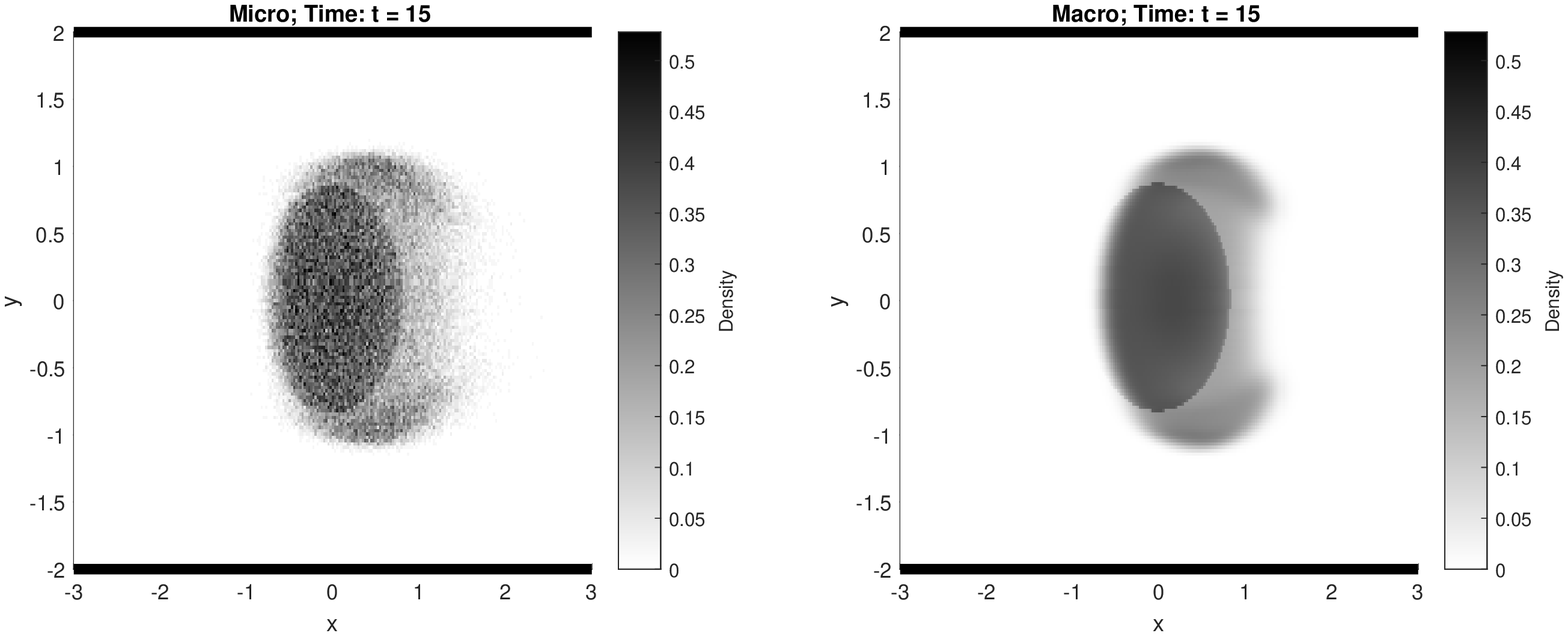}
}\qquad
\subfloat
{
\includegraphics[width = 1\textwidth]{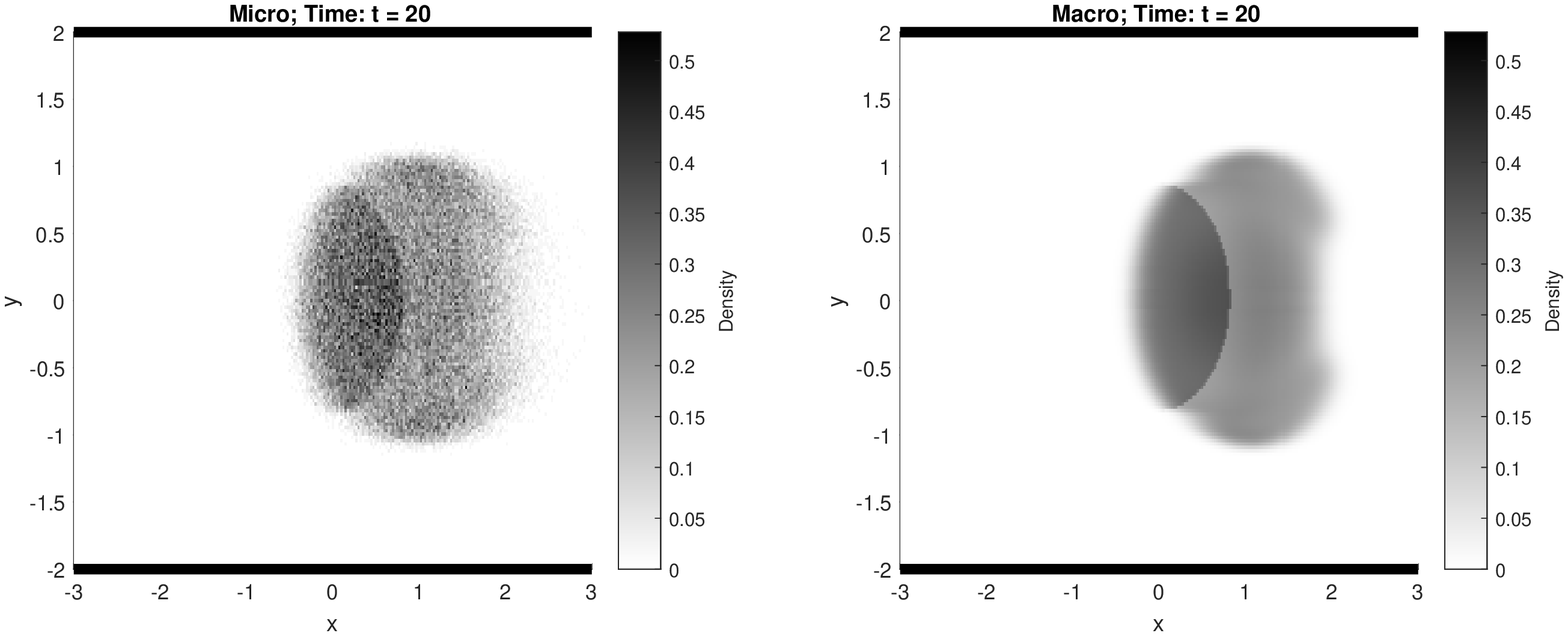}
}
\caption{Densities at different times}
\label{fig:Ex1DensSol}
\end{figure}

Apparently, the results from the microscopic and the macroscopic simulation are quite close, i.e. the density profiles
show the same qualitative behavior in time. This can be also verified
by the mass balance in 
figure \ref{fig:Ex1MassBalance} and the $L^1$ and $L^2$ error in figure \ref{fig:Ex1L1L2Error}.

\begin{figure}[htb!]
\centering
\subfloat{
\includegraphics[width = 0.5\textwidth]{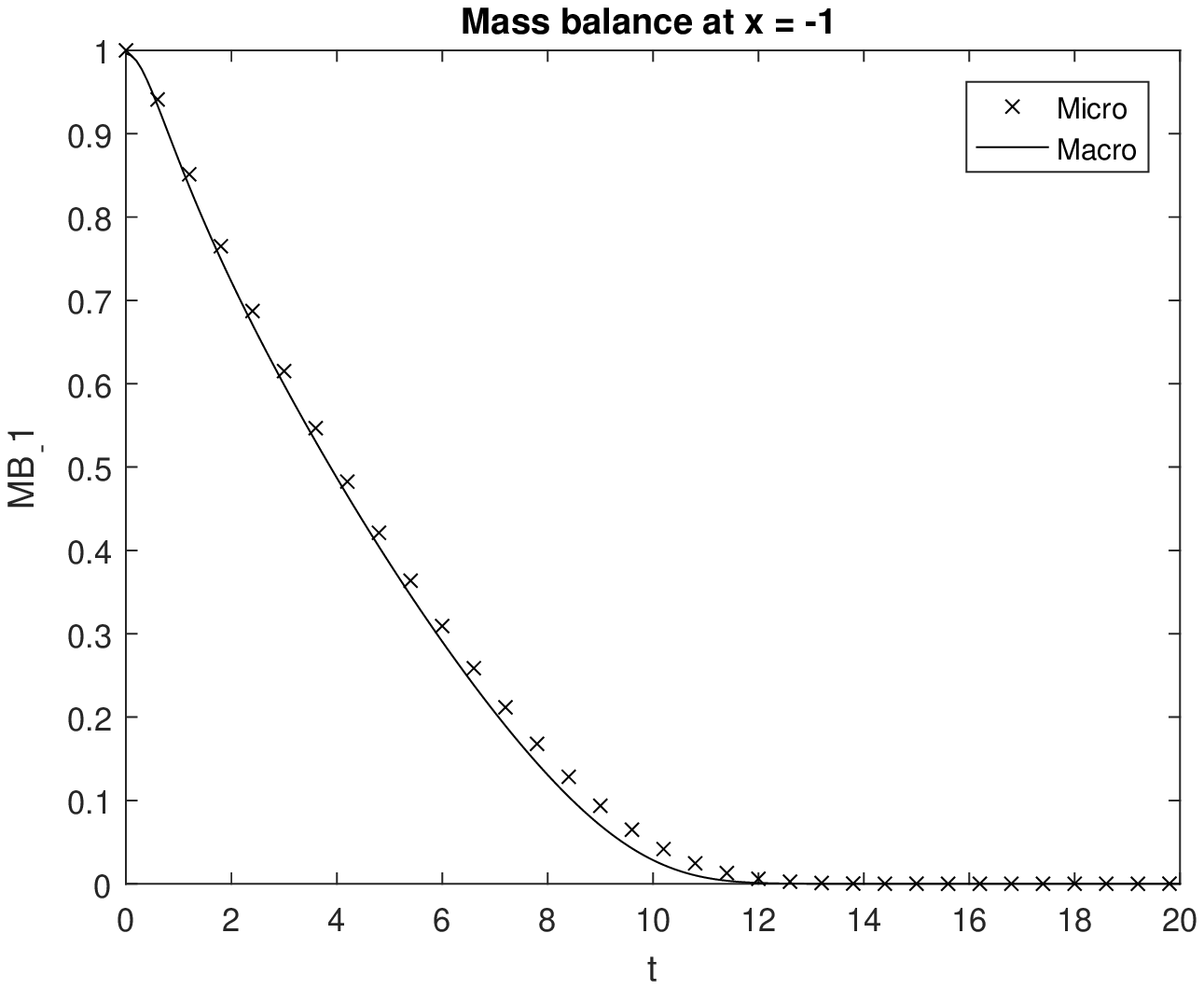}
}
\subfloat{
\includegraphics[width = 0.5\textwidth]{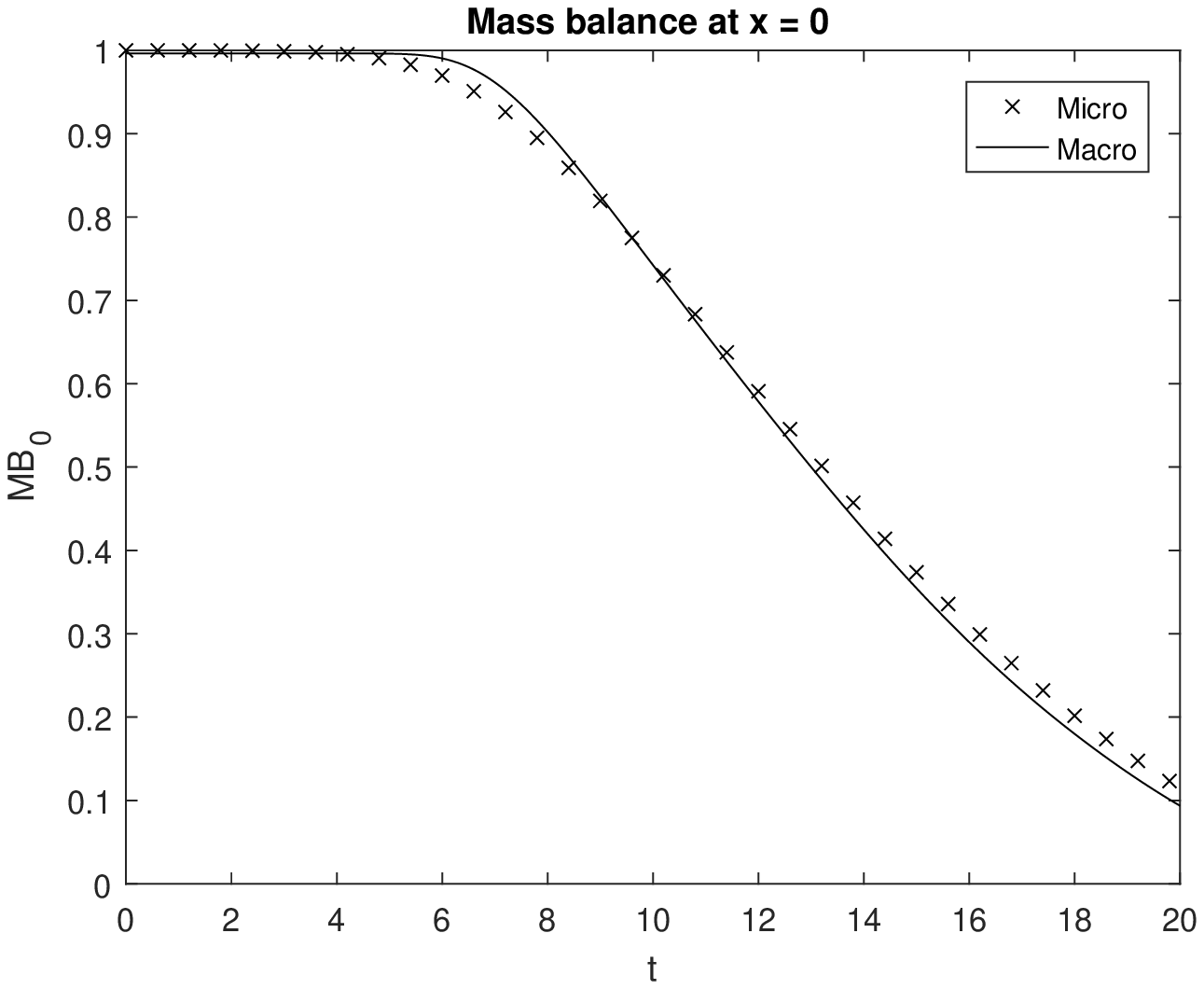}
}
\caption{Mass balances at $x=-1$ and $x=0$}
\label{fig:Ex1MassBalance}
\end{figure}
\begin{figure}[htb!]
\centering
\includegraphics[width = 0.5\textwidth]{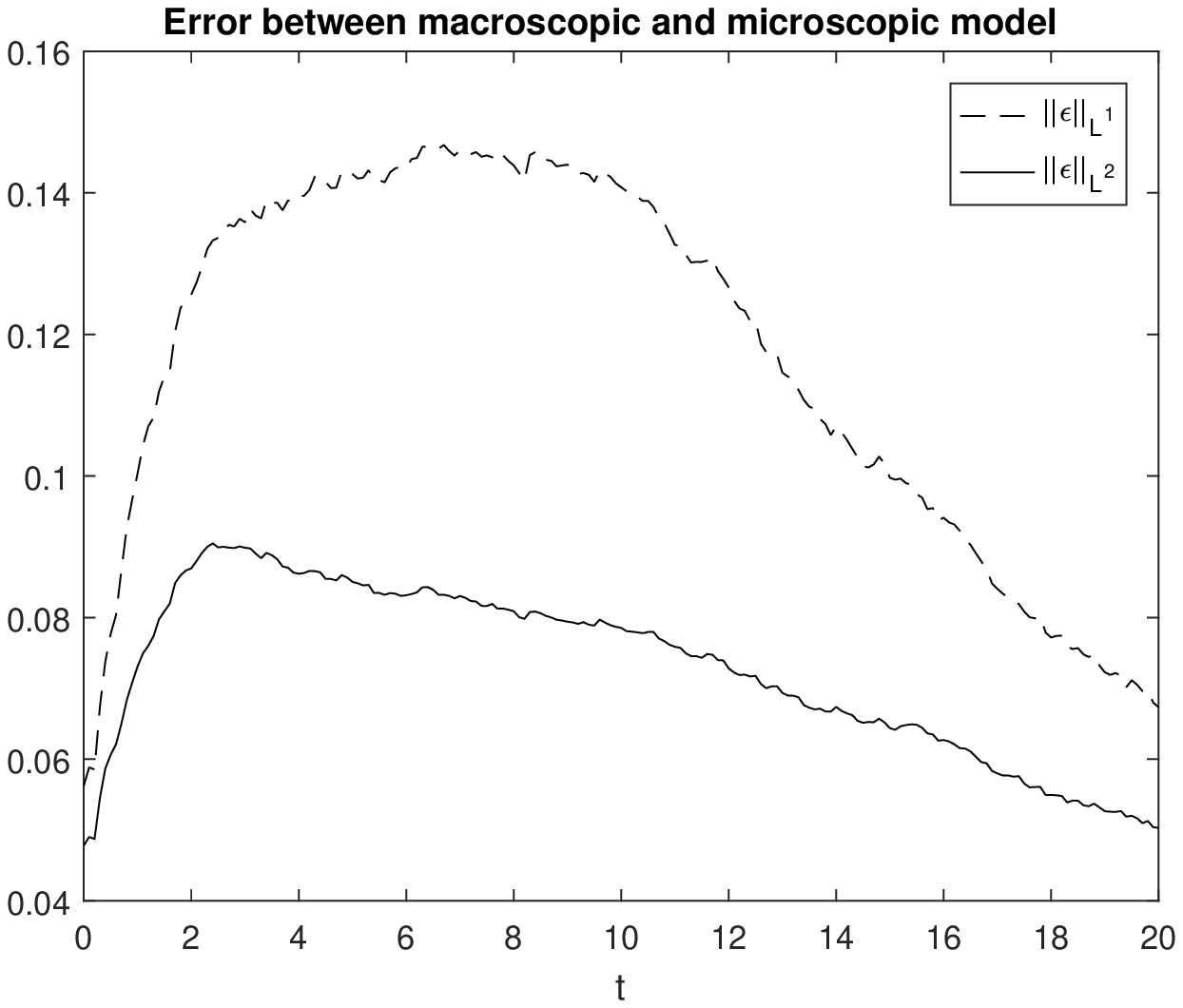}
\caption{$L^1$ and $L^2$ error}
\label{fig:Ex1L1L2Error}
\end{figure}

\subsubsection*{Example 2: Corridor including Bottlenecks}
Different to the first example,
we now focus on a scenario where two obstacles are located oppositely, see  
bold boundaries in figures \ref{fig:Ex2DensSolStoch} and \ref{fig:Ex2DensSolDet}. 
We distinguish the rate functions $\lambda_1$ and $\lambda_2$
to basically cover two relevant scenarios:

\begin{center}\vspace{-8mm}
\begin{tabular}{m{0.5\textwidth}m{0.5\textwidth}}
{
\begin{align*}
\lambda_1(0,x) &= 
\begin{cases}
	1 & \text{ if } x \in [-1,1], \\
	10 &\text{ else,}
\end{cases} 
\end{align*}} &
{
\begin{align*}
\lambda_1(1,x) &= 
\begin{cases}
	1 & \text{ if } x \in [-1,1], \\
	\frac{1}{100} &\text{ else,}
\end{cases} 
\end{align*}
} 
\\[-8mm]
{
\begin{align*}
\lambda_2(0,x) &=  10, 
\end{align*}

} &
{
\begin{align*}
\lambda_2(1,x) &=  \frac{1}{100}.  
\end{align*}
}
\end{tabular}
\end{center} \vspace{-8mm}
The rate function $\lambda_1$ is in-homogeneous in space and divided into two cases. In the case $x \in [-1,1]$ which is exactly between the obstacles (e.g.\ external attractions) pedestrians will stop more often and stay for a longer period.
Otherwise, for $x \in (-\infty,-1)\cup (1,\infty)$, pedestrians behave deterministic since the scaling factor $(1+\tau \lambda(1,x))^{-1}$ in \eqref{eq:closVelo} becomes very close to one and we get a velocity which corresponds to the deterministic model, cf. left column in figure~\ref{fig:Overview}.
On the contrary, the rate function $\lambda_2$ is homogeneous in space but will be used as the deterministic benchmark problem, see again
figure~\ref{fig:Overview}.

We consider the following setup: The destination is again located at $x^D = (100,0)$, the comfort speed is $v^C = 1$ and the relaxation time is supposed to be $\tau = 0.2$. The initial distribution is 
\[\rho^0(B) = \frac{2}{3}\int_B \Ind_{[-2.5,-1]\times[-0.5,0.5]}(x,y) d(x,y).\]
with probability $p_0 = 0.01$ to stop in the beginning. Therefore, the initial conditions for the macroscopic model read
\[g_0(x,y) = p_0  \frac{2}{3}\Ind_{[-2.5,-1]\times[-0.5,0.5]}(x,y) \text{ and } g_1(x,y) = (1-p_0)\frac{2}{3}\Ind_{[-2.5,-1]\times[-0.5,0.5]}(x,y).\]

The simulation is performed with step-sizes $\Delta x = \Delta y = \frac{1}{40}$, $N = 100$ pedestrians and $M = 1000$ Monte-Carlo iterations. We have used the same PC as before and the required computation times are 3006 seconds ($\lambda_1$), 3205 seconds ($\lambda_2$) for the microscopic model and 485 seconds ($\lambda_1$), 495 seconds ($\lambda_2$) for the macroscopic model.

In figures \ref{fig:Ex2DensSolStoch}--\ref{fig:Ex2DensSolDet} we plot the densities for different rate functions. 
The specular reflection type boundaries seem to work well and the pedestrians are rerouted along the obstacles.
For the choice of $\lambda_1$ in figure \ref{fig:Ex2DensSolStoch}, where the pedestrians tend to stop between the obstacles,   
we observe the expected queuing behavior. In contrast for the deterministic setting $\lambda_2$, higher densities in front of the obstacles result from the spatial conditions only and are not caused by stop-or-go decisions, cf. figure \ref{fig:Ex2DensSolDet}.  

The influence of different rate functions is also reflected by the comparison of the mass balances in figure \ref{fig:Ex2MB}. 
The first scenario, the bottleneck situation, reduces significantly the mean velocity of the pedestrians
and hence leads to longer travel times, see left picture in figure \ref{fig:Ex2MB}. 
After approximately 14 time units all pedestrians crossed the cut at $x = 1$ 
while for the deterministic scenario only 2.5 time units are needed, cf. right picture in figure \ref{fig:Ex2MB}.

Finally, we stress that the error between the microscopic and macroscopic model is again rather small, as figure \ref{fig:Ex2Errors} illustrates, and is not affected by the different rate functions. 

\begin{figure}[htb!]
\centering
\subfloat
{
\includegraphics[width = 1\textwidth]{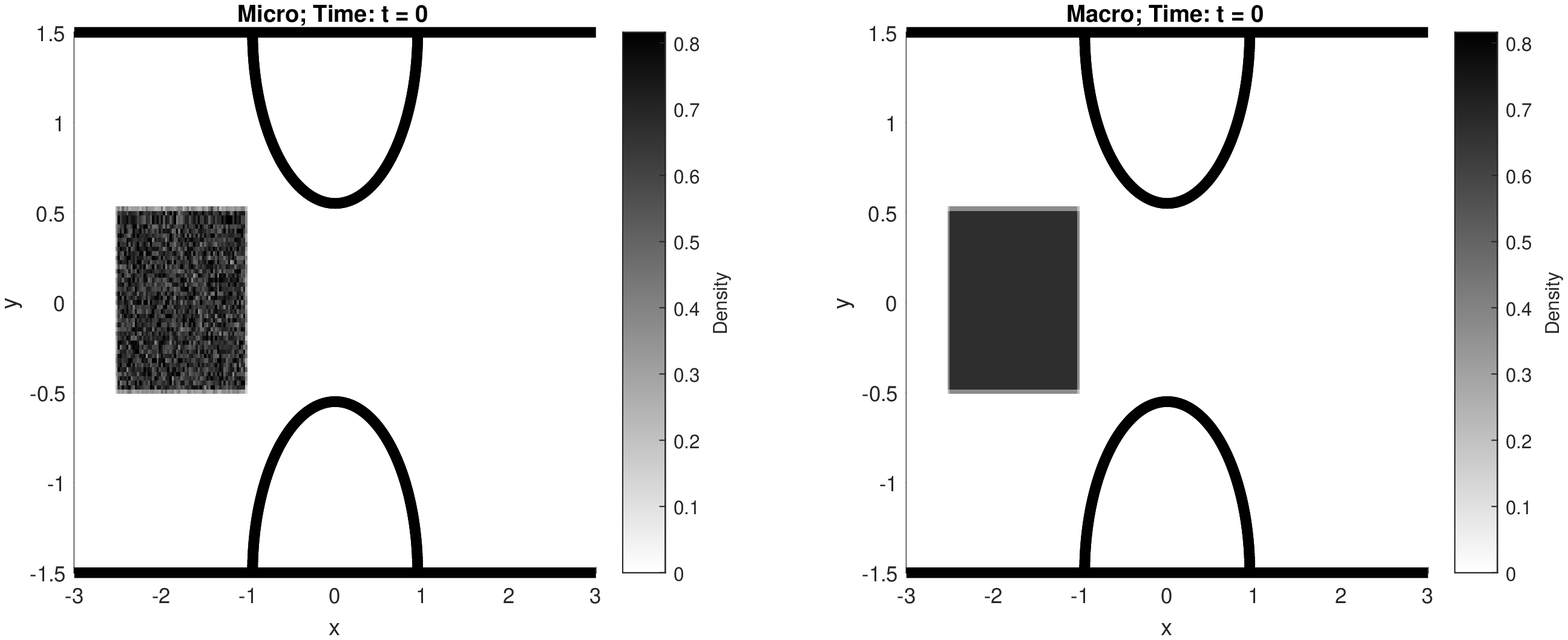}
}\qquad
\subfloat
{
\includegraphics[width = 1\textwidth]{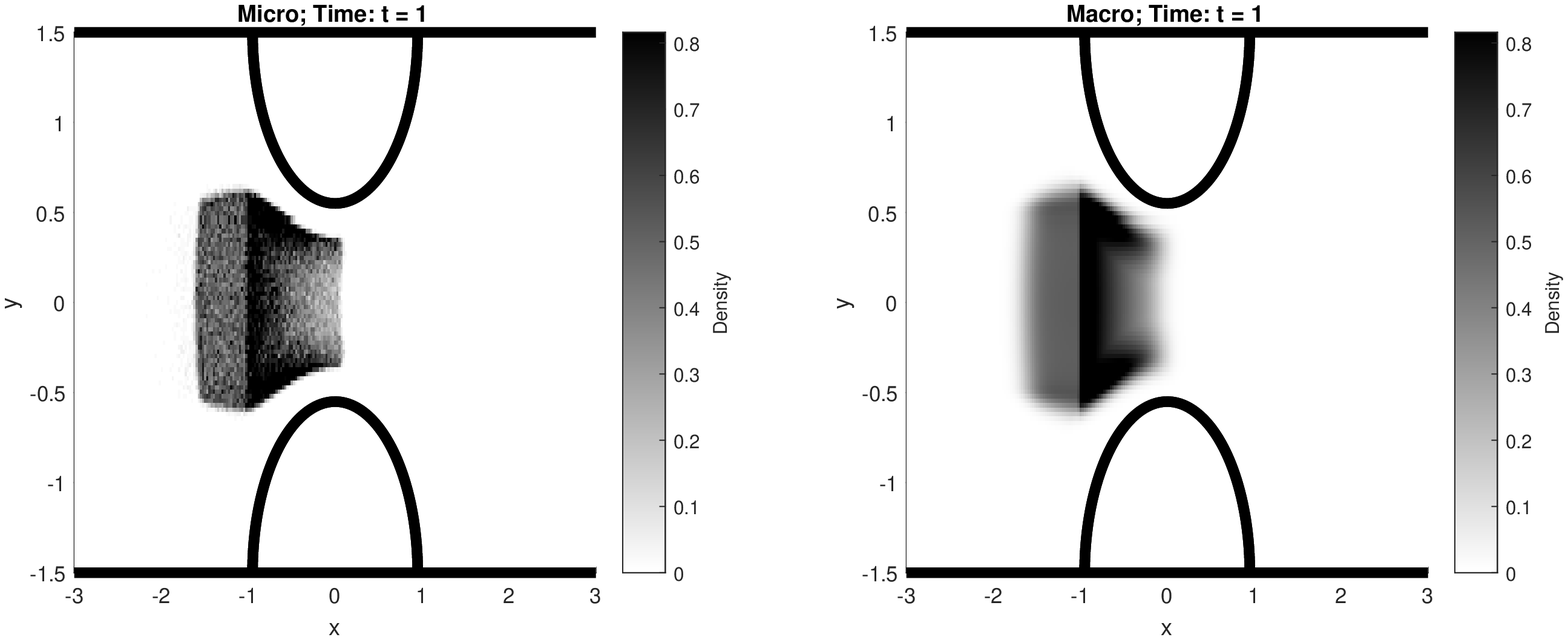}
}\qquad
\subfloat
{
\includegraphics[width = 1\textwidth]{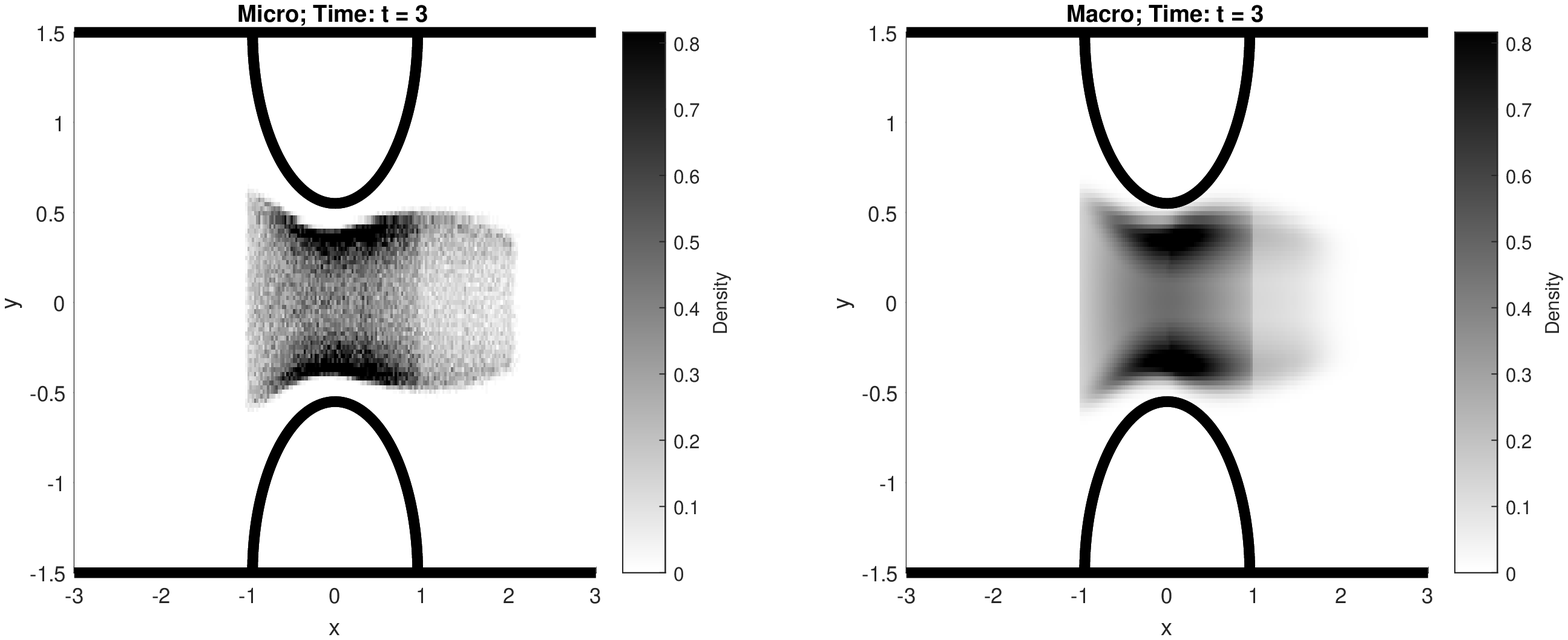}
}\qquad
\subfloat
{
\includegraphics[width = 1\textwidth]{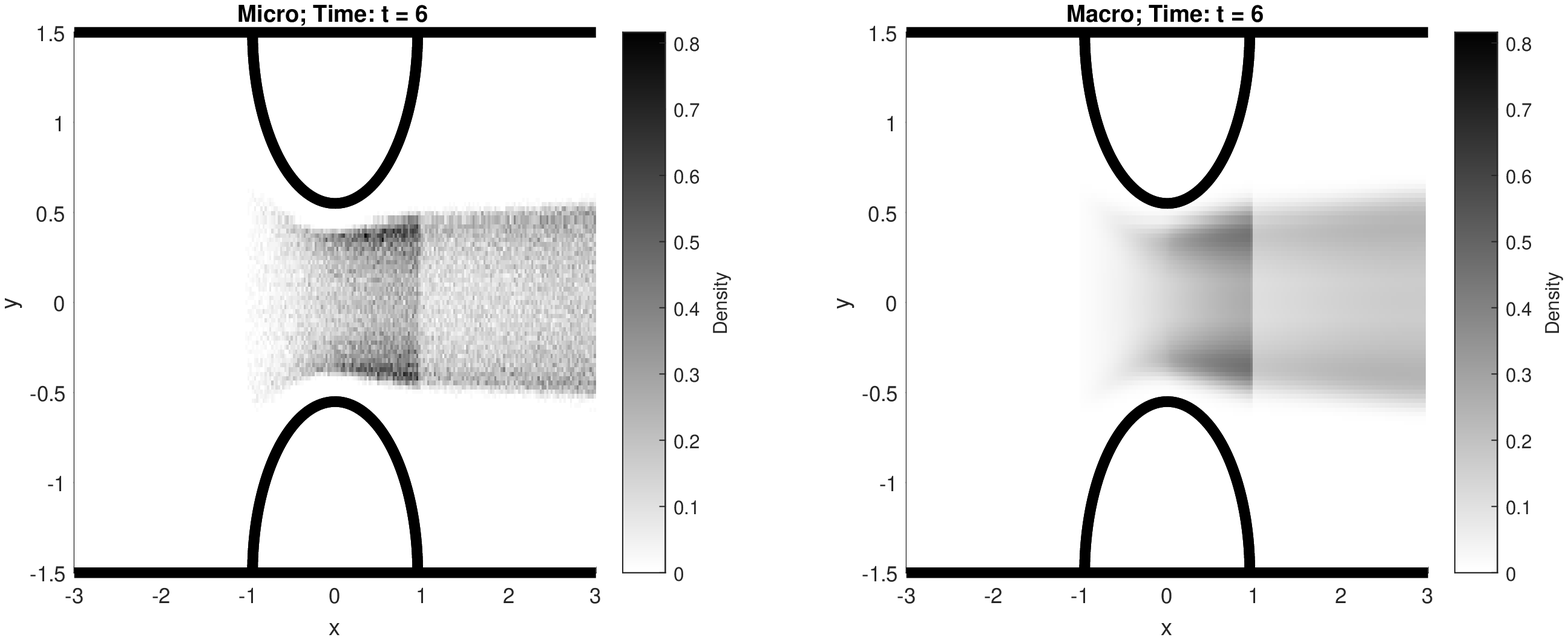}
}
\caption{Densities for $\lambda_1$ at different times}
\label{fig:Ex2DensSolStoch}
\end{figure}

\begin{figure}[htb!]
\centering
\subfloat
{
\includegraphics[width = 1\textwidth]{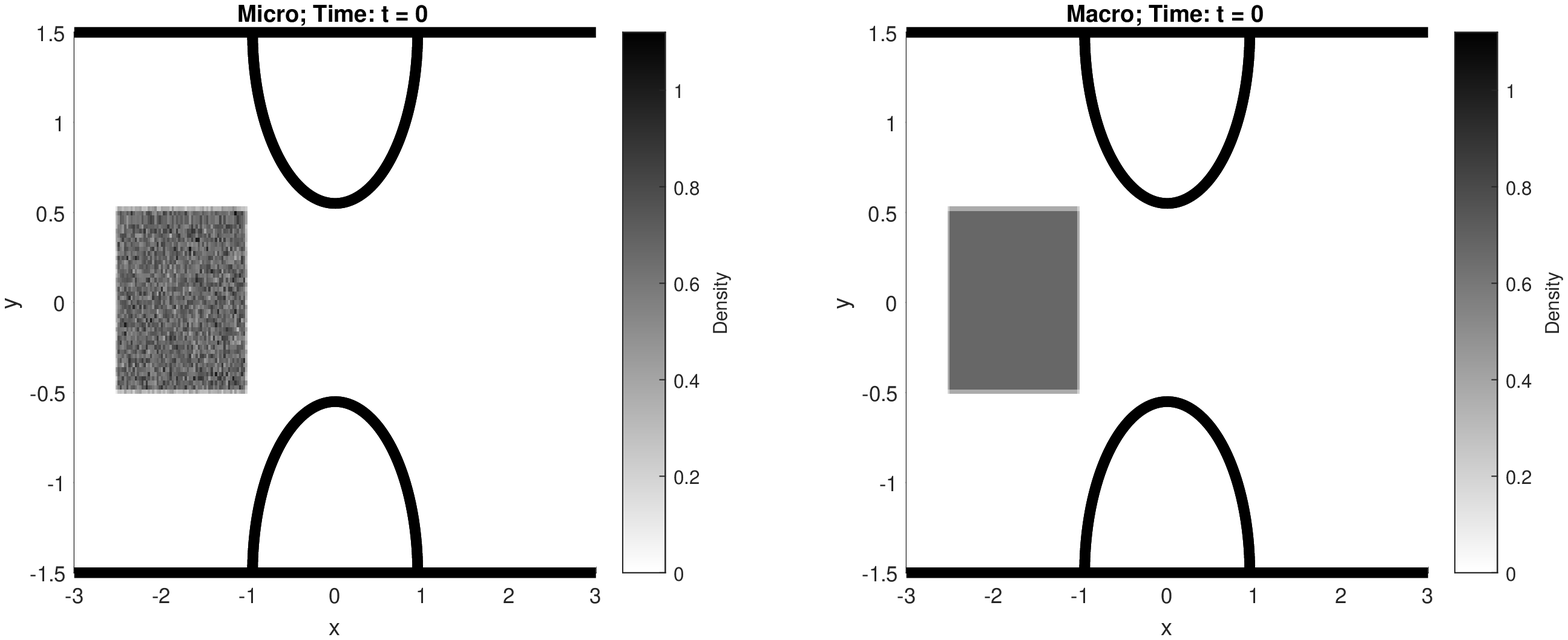}
}\qquad
\subfloat
{
\includegraphics[width = 1\textwidth]{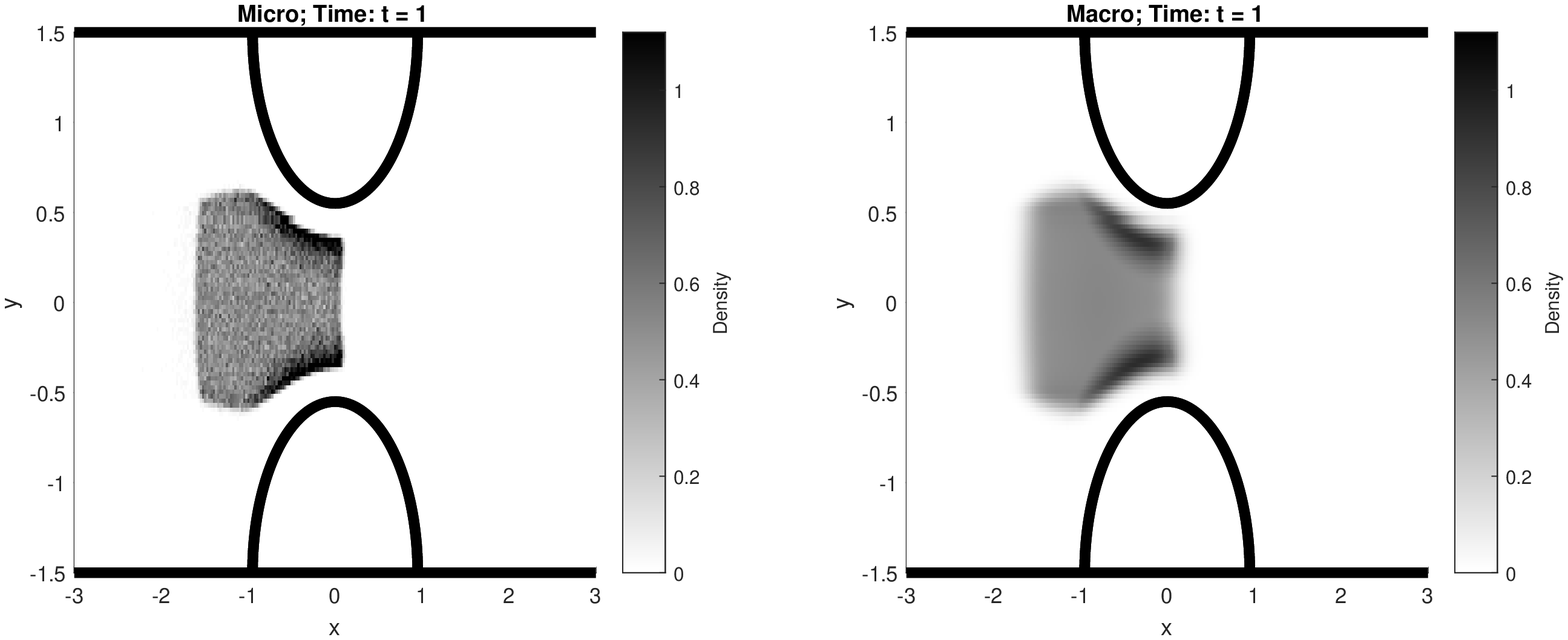}
}\qquad
\subfloat
{
\includegraphics[width = 1\textwidth]{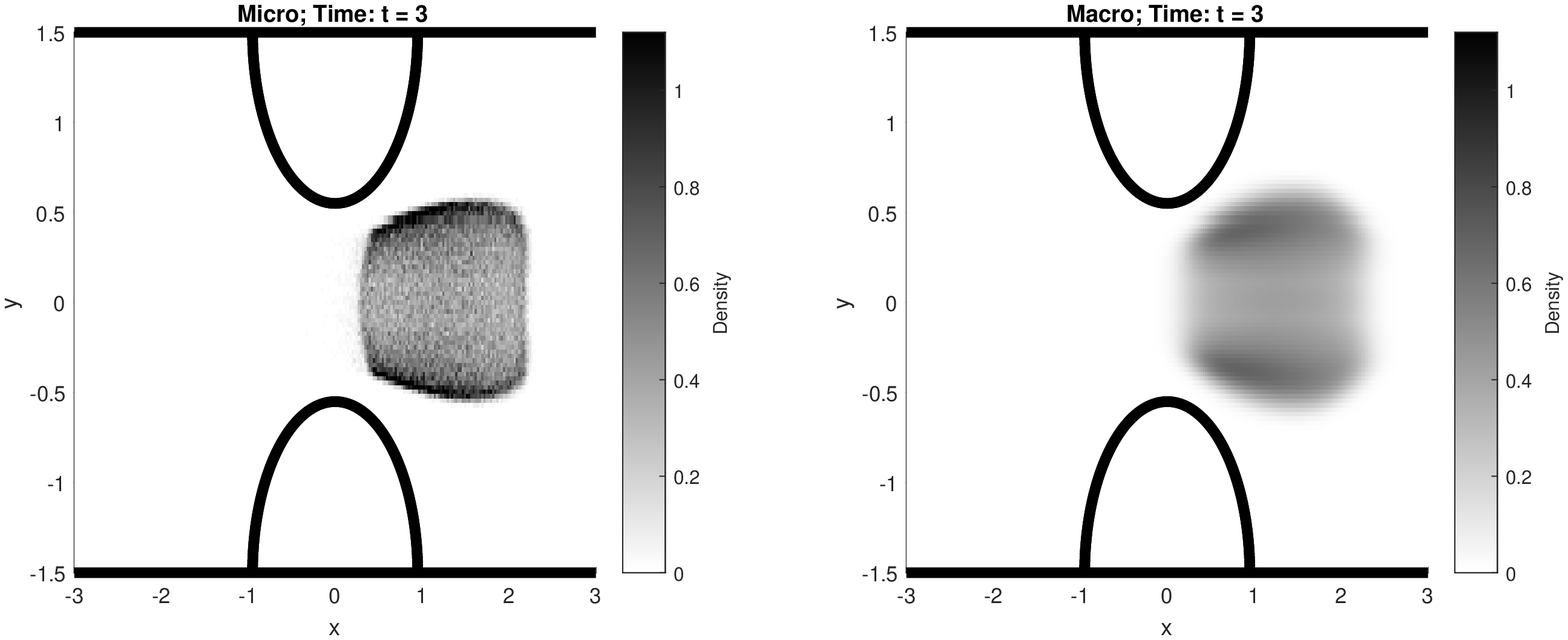}
}
\caption{Densities for $\lambda_2$ at different times}
\label{fig:Ex2DensSolDet}
\end{figure}

\begin{figure}[htb!]
\centering
\subfloat[Mass balance with $\lambda_1$]
{
\includegraphics[width = 0.5\textwidth]{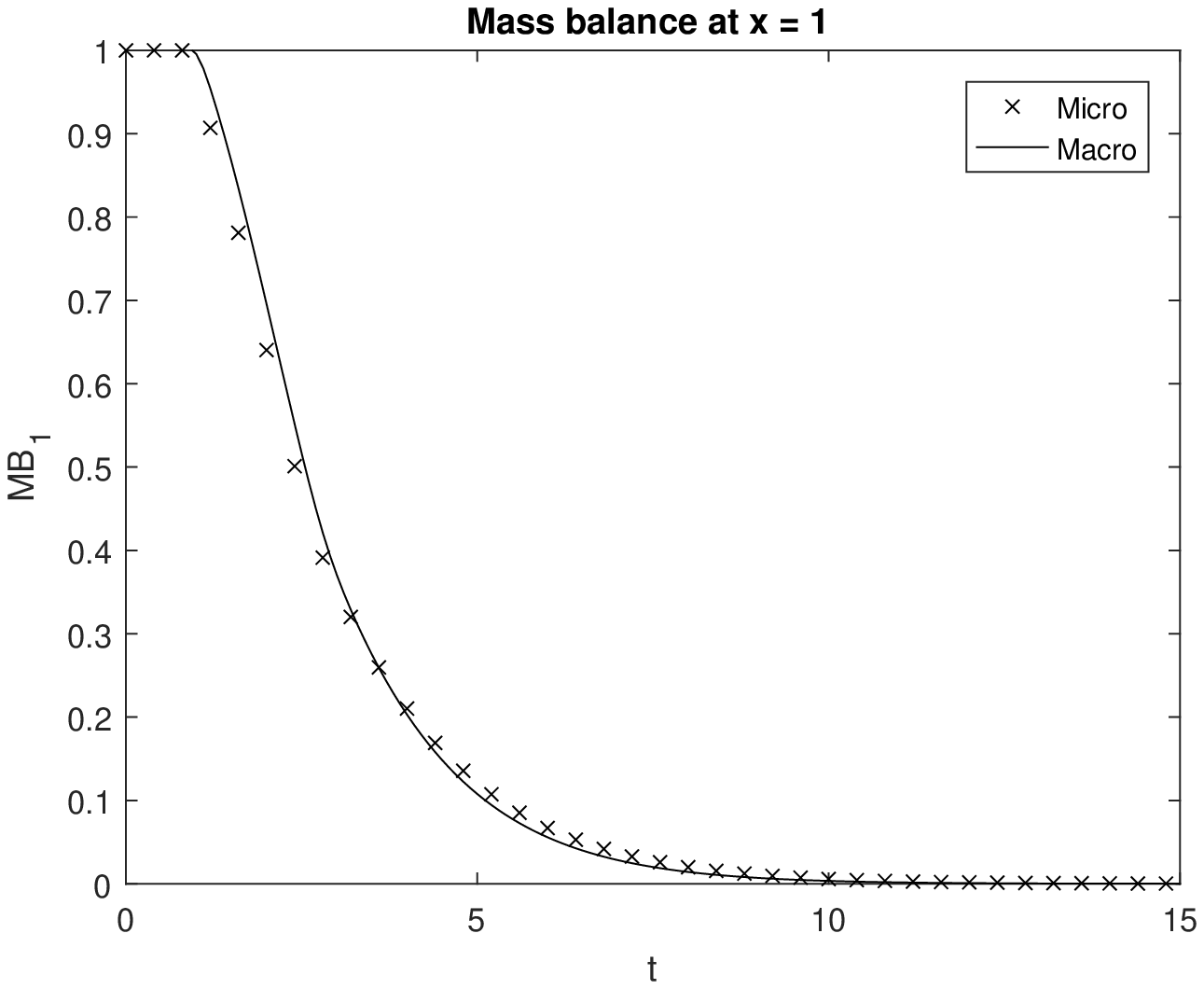}
}
\subfloat[Mass balance with $\lambda_2$]
{
\includegraphics[width = 0.5\textwidth]{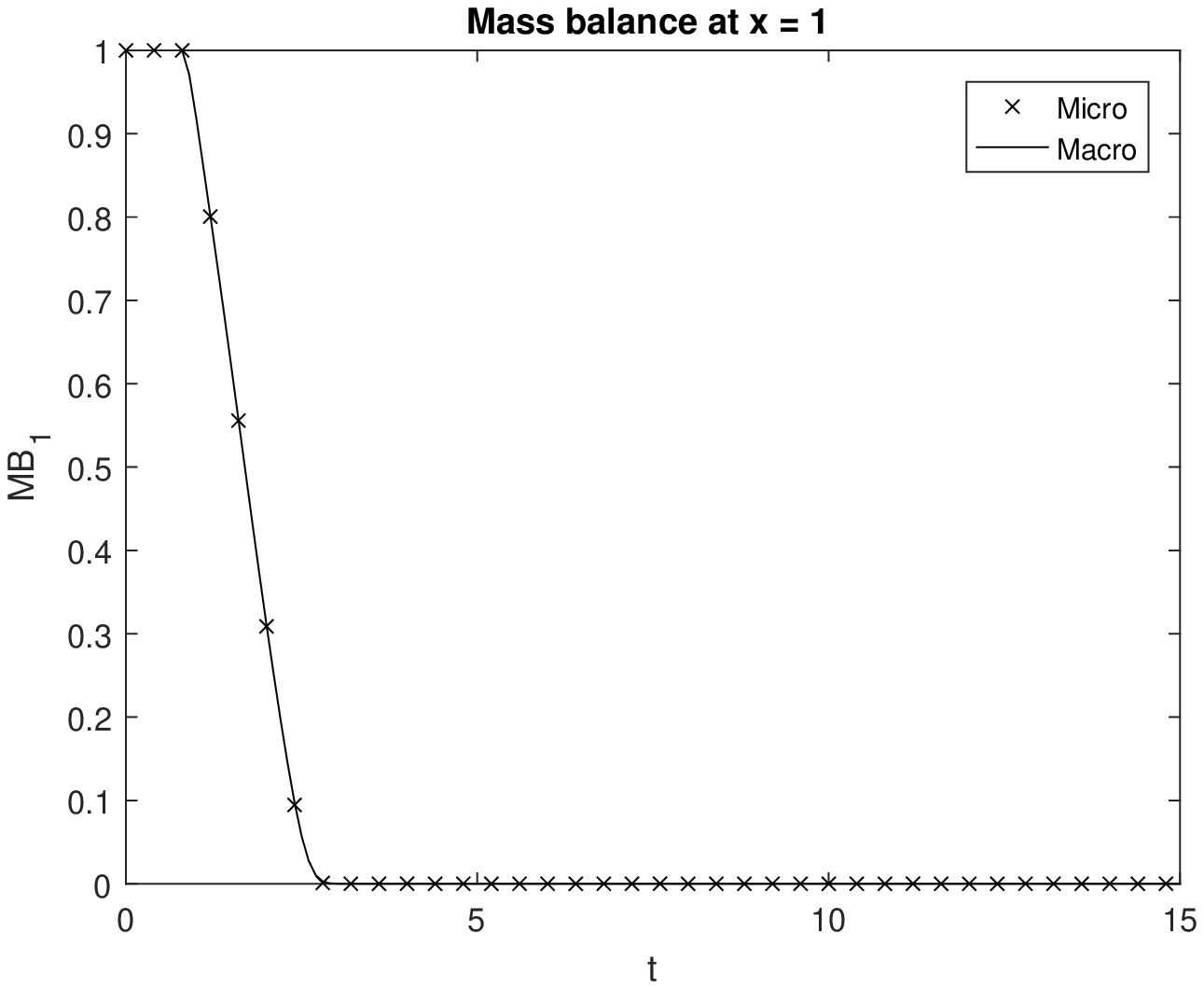}
}
\caption{Mass balances at $x = 1$}
\label{fig:Ex2MB}
\end{figure}

\begin{figure}
\centering
\subfloat[Errors with $\lambda_1$]
{
\includegraphics[width = 0.5\textwidth]{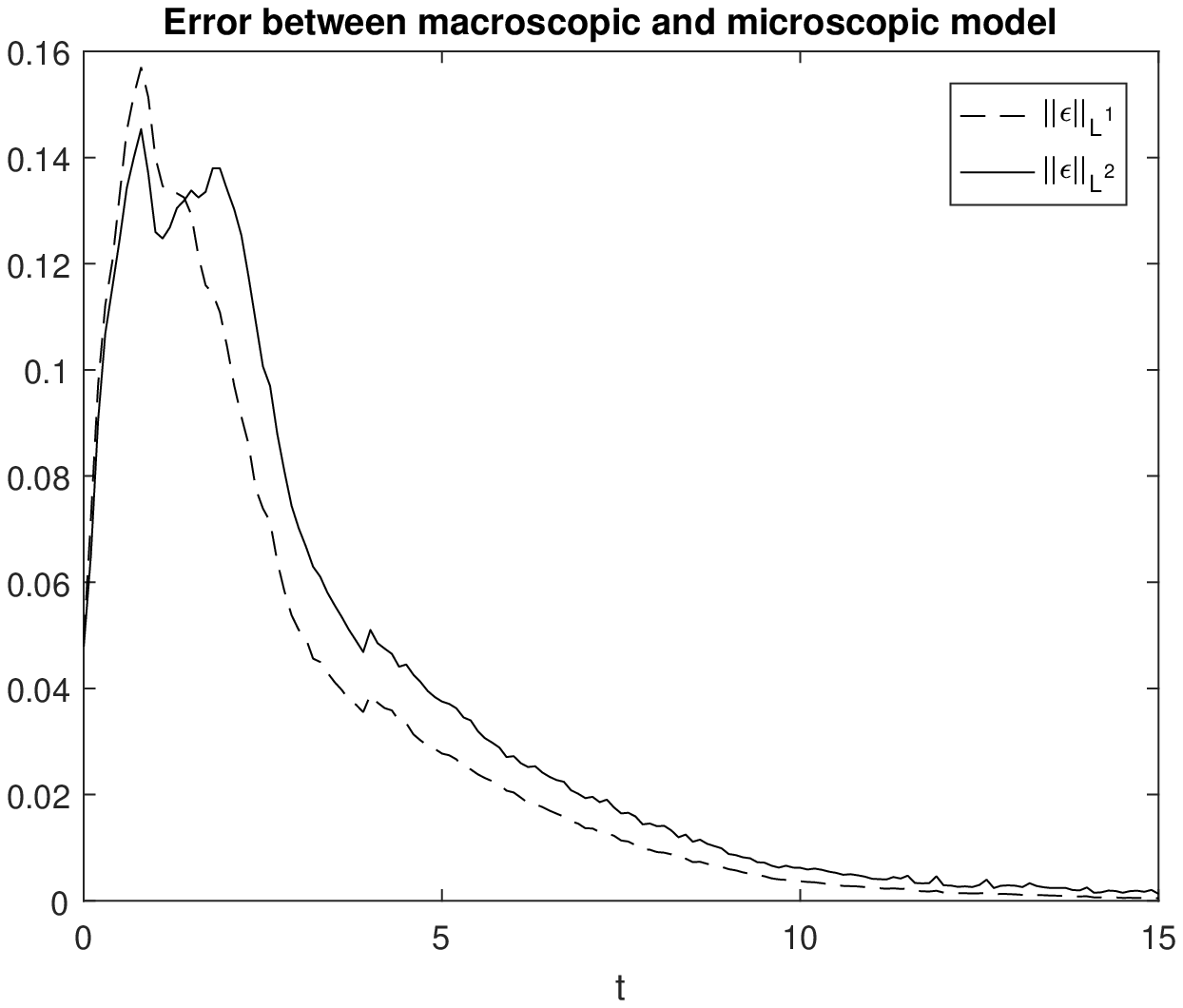}
}
\subfloat[Errors with $\lambda_2$]
{
\includegraphics[width = 0.5\textwidth]{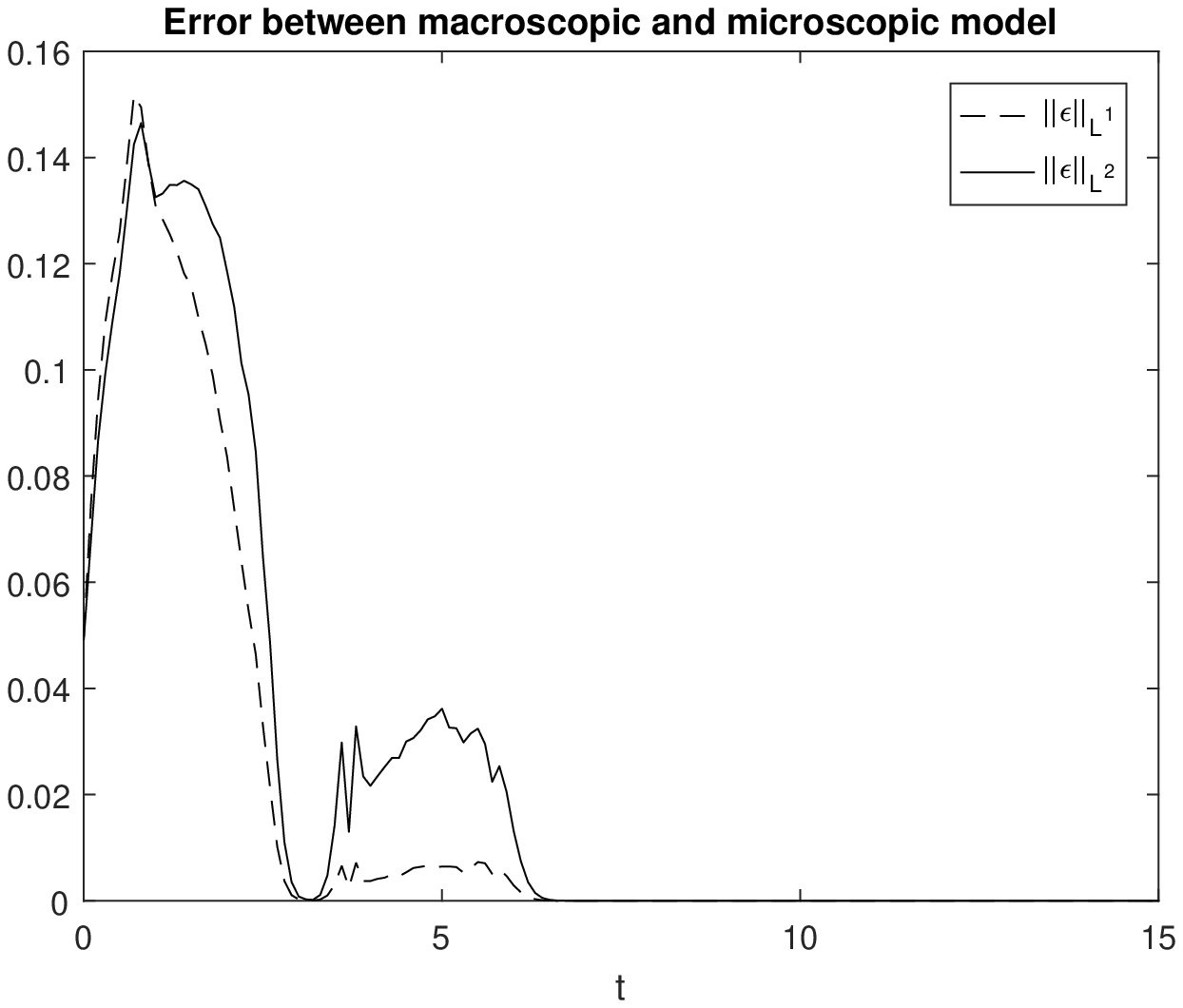}
}
\caption{$L^1$ and $L^2$ errors}
\label{fig:Ex2Errors}
\end{figure}

%% file: conclusion.tex
\section*{Conclusion}

We have introduced a stochastic microscopic pedestrian model with specular reflection type boundary conditions and proved its existence. Based on the microscopic model, we have derived a kinetic formulation under a mean field assumption in terms of measures. 
A non-standard closure assumption has been established to extract a first order macroscopic model given by
a system of conservation laws. 
Numerical comparisons of the microscopic and macroscopic pedestrian flow model have been performed
to demonstrate the same dynamical behavior for both modeling approaches.

Future work might include the investigation of a stochastic model hierarchy, where the Eikonal equation (instead of the
destination force) is used to determine the shortest travel time to the desired destination, see \cite{Etikyala2014}.
The derivation of a second order macroscopic model seems to be also very appealing in this context.